\documentclass[11pt]{amsart}

\usepackage{amsmath}
\usepackage{amsfonts}
\usepackage{amssymb}
\usepackage{ytableau}
\usepackage{amsthm}
\usepackage{tikz}
\usetikzlibrary{tikzmark,decorations.pathreplacing}
\usepackage{hyperref}
\usepackage{color}

\usepackage{graphicx}

\usepackage{tikz-cd}
\usepackage[numbers]{natbib}
\usepackage{appendix}
\usetikzlibrary{tqft}
\usepackage{mathtools}
\usepackage{enumerate}
\usepackage{multicol}
\usepackage{seqsplit}

\newtheorem{theorem}{Theorem}[section]

\newtheorem{lem}[theorem]{Lemma}
\newtheorem{cor}[theorem]{Corollary}

\newtheorem{prop}[theorem]{Proposition}
\newtheorem{rem}[theorem]{Remark}

\newtheorem{lemma}[theorem]{Lemma}
\newtheorem{corollary}[theorem]{Corollary}
\newtheorem{definition}[theorem]{Definition}
\newtheorem{proposition}[theorem]{Proposition}

\newcommand{\bb}{\mathbb}
\newcommand{\ca}{\mathcal}
\newcommand{\fr}{\mathfrak}

\newcommand{\Z}{\ensuremath{\mathbb{Z}}}
\newcommand{\Q}{\ensuremath{\mathbb{Q}}}
\newcommand{\C}{\ensuremath{\mathbb{C}}}

\renewcommand{\P}{\ensuremath{\mathrm{P}}}
\newcommand{\bP}{\ensuremath{\mathbb{P}}}

\renewcommand{\O}{\ensuremath{\mathcal{O}}}
\newcommand{\rO}{\ensuremath{\mathrm{O}}}
\newcommand{\cS}{\ensuremath{\mathcal{S}}}

\newcommand{\bS}{\ensuremath{\mathbb{S}}}
\newcommand{\bG}{\ensuremath{\mathbb{G}}}
\newcommand{\fX}{\ensuremath{\fr{X}}}

\newcommand{\bt}{\mathbf{t}}

\newcommand{\ra}{\rightarrow}

\newcommand{\bx}{\mathbf{x}}
\newcommand{\by}{\mathbf{y}}
\newcommand{\bp}{\mathbf{p}}

\newcommand{\git}{\mathbin{
\mathchoice{/\mkern-6mu/}
{/\mkern-6mu/}
{/\mkern-5mu/}
{/\mkern-5mu/}}}

\DeclareMathOperator{\shom}{\ca{H}om}

\DeclareMathOperator{\Sym}{\mathrm{Sym}}

\newcommand{\ev}{\mathrm{ev}}
\newcommand{\wev}{\widetilde{\ev}}

\newcommand{\vir}{\mathrm{vir}}

\newcommand{\ovir}{\ca{O}^{\mathrm{vir}}}
\newcommand{\Hom}{\operatorname{\ca{H}\! \mathit{om}}}
\newcommand{\GLr}{\mathrm{GL}_r(\bb{C})}
\newcommand{\glr}{\mathfrak{gl}_r(\bb{C})}
\newcommand{\kg}{K_{T}^0}
\newcommand{\kG}{K_0^{T}}
\newcommand{\spec}{\operatorname{Spec}}
\newcommand{\qk}{\mathrm{QK}}

\newcommand{\ud}{\underline{d}}
\newcommand{\hc}{\widehat{C}}

\newcommand{\prk}{\P_{r,k}}

\DeclareMathOperator{\cO}{\mathcal{O}}
\DeclareMathOperator{\quot}{\mathsf{Quot}}
\newcommand{\Quot}{\quot_d(\bP^1, N,r)}
\newcommand{\QuotC}{\quot_{d}(C,N,r)}
\newcommand{\QuotCB}{\fr{Quot}_{d}(\fr{C}/\fr{B},N,r)}
\newcommand{\QuotCPB}{\fr{Quot}_{d}(\fr{C}'/\fr{B},N,r)}

\newcommand{\QuotCpm}{\fr{Quot}_{d}(\fr{C}_\pm/\ca{T},N,r)}
\newcommand{\QuotCt}{\fr{Quot}_{d}(\fr{C}_\sim/\ca{T}_\sim,N,r)}

\newcommand{\md}{\ca{M}^d}
\newcommand{\mdo}{\ca{M}^d_0}

\newcommand{\mdpm}{\ca{M}^d_\pm}
\newcommand{\mdt}{\ca{M}^d_\sim}

\newcommand{\cdegn}{C_{\rm deg}[n]}

\DeclareMathOperator{\rank}{\operatorname{rank}}

\newcommand{\grass}{\mathrm{Gr}(r,N)}

\newcommand{\ms}{\overline{M}_{0,n}(X,d)}
\newcommand{\qgs}{Q^{0+}_{C,n}(X,d)}
\newcommand{\msmm}{Q^{0+,0+}_{C,m\mid n-m}(X,d)}

\newcommand{\QGS}{QG^{\epsilon=0+}_{0,1}(X,d)}
\newcommand{\epqm}{Q^{\epsilon}_{0,n}(X,d)}
\newcommand{\stableq}{Q^{\epsilon=0+}_{0,n}(X,d)}

\usepackage{xcolor}
\begin{document}
\baselineskip=15pt
	\title[Quantum $K$-invariants via Quot schemes I]{Quantum $K$-invariants via Quot schemes I
	}
	
\author[S.~Sinha]{SHUBHAM SINHA}
\address{International Centre for Theoretical Physics, Trieste}
\email{ssinha1@ictp.it}
\author[M.~Zhang]{Ming Zhang}
\address{Department of Mathematics, University of California, San Diego}
\email{miz017@ucsd.edu}

\maketitle

\begin{abstract}
We study the virtual Euler characteristics of sheaves over Quot schemes of curves,
establishing that these invariants fit into a topological quantum field theory (TQFT) valued in $\mathbb{Z}[[q]]$. We show that the three-pointed genus-zero $K$-theoretic stable map invariants of the Grassmannian coincide with the genus-zero $K$-theoretic invariants defined via the Quot scheme. Utilizing Quot scheme compactifications alongside the TQFT framework,
we derive presentations of the small quantum $K$-ring of the Grassmannian. Our approach offers a new method for finding explicit formulas for quantum $K$-invariants.
\end{abstract}

\section{Introduction}

\subsection{Overview}
Quantum $K$-theory was introduced by Givental~\cite{WDVV} and Lee
\cite{Lee} as a $K$-theoretic generalization of Gromov--Witten theory.
Let $X$ be a projective manifold.
Quantum $K$-invariants (or sometimes referred to as the $K$-theoretic Gromov--Witten invariants)
are defined as Euler characteristics of $K$-theory classes on the Kontsevich's moduli spaces of stable maps to $X$.
Using 2-pointed and 3-pointed, genus zero quantum $K$-invariants of $X$, one can define an associative product, called the quantum product,
on the $K$-group of $X$.
The resulting algebra is called the (small) quantum $K$-ring.

Let $\grass$ be the Grassmann variety of $r$-planes in $\C^{N}$.
The small quantum $K$-ring of Grassmannians was first studied by Buch and Mihalcea in~\cite{Buch-Mihalcea}. Buch--Mihalcea's results rely on the analysis of the geometry of the Kontsevich moduli space $\overline{M}_{0,3}(\grass,d)$, which compactifies the moduli space $\mathrm{Mor}_{d}(\bP^{1},\grass)$ of degree-$d$ morphisms from $\bP^{1}$ to $\grass$.
Note that the moduli space $\mathrm{Mor}_{d}(\bP^{1},\grass)$ has a simpler compactification, the Grothendieck’s Quot
scheme~\cite{Grothenieck1}.
In the early study of \emph{quantum cohomology} of Grassmannians, many results have been obtained by using the Quot scheme compactification~\cite{Bertram3, Bertram2,BDW}.

In this paper, we show that virtual Euler characteristics over Quot schemes of any genus $g$ curve form a $\bb{Z}[[q]]$-valued $(1+1) $-Topological Quantum Field Theory (TQFT). Additionally, we establish that the quantum $K$-ring
$\qk (\grass)$ can be fully reconstructed using the Quot schemes of $\bb{P}^1$.
In particular, our approach provides a new way to derive a ring presentation for $\qk(\grass)$. We observe that fundamental properties of $\qk(\grass)$, such as finiteness and $S_3$-symmetry
of the structure constants, are geometrically elucidated using this approach. 

In part II \cite{SinhaZhang2},
we establish explicit formulas for the Euler characteristics over Quot schemes of curves,
similar to the Vafa--Intriligator formula for cohomological invariants.
For the genus zero case, we employ torus localization for calculation, inspired by ~\cite{OpreaShubham} in $K$-theory (see also~\cite{Marian-Oprea} for cohomology calculations).
For higher-genus cases, we utilize the TQFT structure developed here. 
In this paper, some results critically depend on the vanishing results in Theorem~\ref{thm:vanishing} from~\cite{SinhaZhang2}.

\subsection{Quantum $K$-ring of the Grassmannian}
Throughout this paper, we set $X=\grass$ and $k=N-r$. Let
$
\P_{r,k}
$ be the set of partitions contained in the $r\times k$ rectangular partition. The complement partition for $\lambda=(\lambda_{1},\dots,\lambda_r)\in \P_{r,k}$ is given by $\lambda^{*}=(k-\lambda_r,\dots,k-\lambda_1)$. 

Let $K(X)$ be the Grothendieck group of coherent sheaves (or, equivalently,
vector bundles since $X$ is smooth).
For any partition $\lambda\in\P_{r,k}$, let $X_{\lambda}$ denote the Schubert cycle for a fixed complete flag of $\mathbb{C}^N$
(see Section~\ref{sec:K-theory} for precise definitions).
The $K$-theory classes of the structure sheaves of these Schubert cycles
$\cO_\lambda:=[\cO_{X_\lambda}]$ form a $\mathbb{Z}$-linear basis for $K(X)$. The structure constants $c_{\lambda,\mu}^{\nu}$ for the product $$\cO_\lambda\cdot \cO_\mu =\sum_{\nu\in \P_{r,k}}^{} c_{\lambda,\mu}^{\nu}\cO_{\nu}$$
in $K(X)$ are given by the $K$-theoretic Littlewood-Richardson rule in~\cite{Buch},
which involves counting set-valued tableaux.

The $K$-group of Grassmannian admits another natural basis consisting of Schur bundles. Let $S\subset \ca{O}_{X}\otimes \bb{C}^N$ be the tautological subbundle over $X$.
For any partition $\lambda$, we have a corresponding vector bundle $\bb{S}^{\lambda}(S)$, where $\bb{S}^{\lambda}$ is the Schur functor associated to $\lambda$. For example,
$\bb{S}^{(m)}(S) = \Sym^m S $ and $\bb{S}^{(1)^m}(S) = \wedge^m S $ where $(1)^m = (1,\dots,1)$.
Let $\rm{A}=({\rm{A}}_{\lambda}^\nu)_{\lambda,\nu\in\P_{r,k}}$ be the transition matrix between the Schur
basis $\{\bb{S}^{\lambda}(S): \lambda\in \P_{r,k} \}$ and the Schubert basis $\{\ca{O}_\nu :\nu \in \P_{r,k} \}$:
\begin{equation}\label{eq:change_of_basis_Grass}
\ca{O}_\lambda=\sum_{\nu\in \P_{r,k}}^{}{\rm{A}}_{\lambda}^\nu \ \bb{S}^{\nu}(S)
\end{equation}
for any $\lambda\in \P_{r,k}$.

 Next we recall the definition of Givental--Lee's quantum $K$-invariants over $X$. Let $\ms$ be the moduli stack of degree-$d$, genus zero, $n$-pointed stable maps into $X$. Let $\ev_{i}:\ms\rightarrow X$ denote the evaluation map at the $i$-th marking for $1\leq i \leq n$. Then the quantum $K$-invariants are defined by
\begin{equation}
	\label{eq:intro-quantum-k}
\langle
\phi_{1},
\dots,
\phi_{n}
\rangle_{0,n,d}
:=\chi
\bigg(
\ms,
\prod_{i=1}^{n}\ev_{i}^{*}(\phi_{i})
\bigg),
\end{equation}
where $\phi_{1},\dots,\phi_{n}$ are $K$-theory classes on $X$. For partitions $\alpha,\beta,\gamma\in \P_{r,k}$, we define
\begin{equation}\label{eq:def_QK-invariats_structure_sheaves}
F_{\alpha,\beta}:=\sum_{d\ge 0}^{}q^d\langle
\cO_{\alpha},
\cO_{\beta}
\rangle_{0,2,d}\quad \text{and}\quad
F_{\alpha,\beta,\gamma}:
=
\sum_{d\geq0}
q^{d}
\langle
\cO_{\alpha},
\cO_{\beta},
\cO_{\gamma}
\rangle_{0,3,d},
\end{equation}
where $q$ is the Novikov variable.
Here, $(F_{\alpha,\beta})_{\alpha,\beta\in \P_{r,k}}$ is referred to as the \emph{quantized pairing} matrix.
Let $F^{\alpha,\beta}$ denote the entries of its inverse.
The \emph{quantum $K$-ring} of $X$, denoted by $\qk(X)$, is the $\Z[[q]]$-module $K(X)\otimes_{\Z}\Z[[q]]$ equipped with the \emph{quantum $K$-product}
defined by
\begin{equation*}
	\cO_{\lambda}
	\bullet
	\cO_{\mu}
	=\sum_{\nu\in\P_{r,k}}
	N_{\lambda,\mu}^{\nu}
	\cO_{\nu},
\end{equation*}
where the structure constants $N_{\lambda,\mu}^{\nu}
$ are given by
\begin{equation}
	\label
	{eq:intro-structure-constants}
	N_{\lambda,\mu}^{\nu}
	:=
	\sum_{\alpha\in\P_{r,k}}
	F^{\nu,\alpha}
	F_{\alpha,\lambda,\mu}
	.
\end{equation}

According to the general results in~\cite{WDVV, Lee}, the quantum $K$-product is commutative, associative, and a deformation of the tensor product.
Furthermore, the  structure constants $N_{\lambda,\mu}^{\nu}$ are polynomials in $q$ (see \cite{Buch-Mihalcea}),
a property commonly known as the finiteness of the quantum $K$-ring.

\subsection{Quot schemes}
\label{subsec:quot-scheme}

Fix a smooth projective curve $C$ of genus $g$. Let $\quot_d(C,N,r)$ be the Quot scheme
parametrizing
short exact sequences $$0\rightarrow S\rightarrow \O_{C}^{\oplus N}
\rightarrow Q\rightarrow 0$$
of coherent sheaves on $C$, where $S$ has rank $r$ and degree $-d$.
When $C=\bb{P}^1$, the Quot scheme $\Quot$, is a smooth projective scheme and isomorphic to $X$ when $d=0$.
Str\o mme~\cite{Stromme} studied geometric aspects of $\Quot$ using torus localization. In higher genus, $\quot_d(C,N,r)$ is used to compactify the space $Mor_d(C,X)$ of degree-$d$ morphisms from $C$ to $X$ and to calculate intersection numbers (see~\cite{BDW,Bertram2,Bertram3})).
In~\cite{Marian-Oprea}, Marian and Oprea defined virtual fundamental classes for Quot schemes on curves of arbitrary genus and calculated virtual intersection numbers.

Let $\pi_C$ and $\pi$ be the projections from $\quot_d(C,N,r)\times C$ to $C$ and $\quot_d(C,N,r)$, respectively.
There exists a universal short exact sequence 
\[
0\rightarrow
\ca{S}
\rightarrow
\pi_C^*\O_{C}^{\oplus N}
\rightarrow
\ca{Q}
\rightarrow 
0.
\]
of coherent sheaves over $\quot_d(C,N,r)\times C$. For any $p\in C$, we denote by $\ca{S}_{p}$ the restriction of the universal subbundle $\ca{S}$ to $\quot_{d}(C,N,r)\times \{p\}$.

\begin{definition}\label{def:structure_sheaf_quot}
	For any partition $\lambda\in \P_{r,k}$, we define the $K$-theory class
	\[\widetilde{\ca{O}}_{\lambda} = \sum_{\nu\in \P_{r,k}} {\rm{A}}_{\lambda}^{\nu}\  \bS^{\nu}(\ca{S}_p)
	\]
	on $\QuotC$, where $\rm{A}_{\lambda}^{\nu}$ are the coefficients of the transition matrix in~\eqref{eq:change_of_basis_Grass} and $\bb{S}^{\nu}(\ca{S}_p)$ are the Schur functors applied to the vector bundle $\ca{S}_p$. Note that the point $p$ is suppressed from the notation of $\widetilde{\ca{O}}_{\lambda}$. We define $$\widetilde{\ca{O}}_{\nu}^*:= \widetilde{\ca{O}}_{\nu^*} \cdot \det(\ca{S}_p),$$ where $\nu^*$ is the complement partition in the $r\times k$ rectangle.
\end{definition}

 \begin{theorem}\label{thm:Intro_structure_constants}
	For any partitions $\lambda, \mu,\nu\in \P_{r,k}$, the structure constants for the product in the quantum $K$-ring $\qk(\grass)$ satisfy:
	\begin{equation*}
		N_{\lambda,\mu}^{\nu} = \sum_{d\ge0}^{}q^d\chi(\Quot, \widetilde{\ca{O}}_{\lambda}\cdot \widetilde{\ca{O}}_{\mu}\cdot \widetilde{\ca{O}}_{\nu}^*).
	\end{equation*}
\end{theorem}
The above theorem provides a more insightful description of the structure constants. For example, Theorem~\ref{thm:Intro_structure_constants} immediately shows the $S_3$-symmetry $N_{\lambda,\mu}^{\nu} = N_{\lambda,\nu^*}^{\mu^*}$.
This symmetry was established for $q=0$ (ordinary $K$-theory) in~\cite{Vakil} and generally in~\cite[Theorem 5.12]{Buch-Mihalcea}, which is not evident from the original definition~\eqref{eq:intro-structure-constants}.
Furthermore, the vanishing results of the Euler characteristics of Schur bundles $\bb{S}^{\lambda}(\ca{S}_p)$ over high-degree Quot schemes
imply that $N_{\lambda\mu}^{\nu}$ is a polynomial in $q$ of degree at most $\min\{r,N-r\}$ (see Corollary~\ref{cor:finiteness-quantum-product}).

\subsection{Topological Quantum Field Theory}
The Quot scheme $\quot_d(C,N,r)$ may not be smooth when $C\neq \mathbb{P}^1$.
A two-term perfect obstruction theory for $\quot_d(C,N,r)$ was constructed
in~\cite{Marian-Oprea}, yielding a virtual fundamental class via \cite{LiTian, Behrend-Fantechi}
and a virtual structure sheaf through \cite{Lee}.
When $d \gg 0$, the Quot scheme is an irreducible scheme of
the expected dimension (see \cite{BDW,PopaRoth}),
and the virtual fundamental class and the virtual structure sheaf, denoted $\mathcal{O}_{\quot}^{\text{vir}}$,
equal the ordinary fundamental class and structure sheaf $\mathcal{O}_{\quot}$, respectively. 

For any $K$-theory class $\phi$ on $\quot_d(C,N,r)$, we define the virtual Euler characteristic by
\[
\chi^{\vir}\big(\quot_d(C,N,r), \phi\big) := \chi\big(\quot_d(C,N,r), \mathcal{O}_{\quot}^{\text{vir}}\cdot \phi\big).
 \]
\begin{definition}\label{defn:intro-TQFT-invariants}
For any tuples of partitions $\underline{\lambda}=( \lambda^1,\dots ,\lambda^{s})$ and $\underline{\nu}=(\nu^1,\dots ,\nu^{t})$ in $\P_{r,k}$, we define the power series
\[
N(g)_{\underline{\lambda}}^{\underline{\nu}} := \sum_{d\ge 0}^{} q^d \chi^{\vir}(\quot_d(C,N,r),\widetilde{\ca{O}}_{\lambda^1}\cdots \widetilde{\ca{O}}_{\lambda^{s}}\cdot \widetilde{\ca{O}}_{\nu^1}^* \cdots\widetilde{\ca{O}}_{\nu^t}^* ),
\]
where $\widetilde{\ca{O}}_\lambda$ and $\widetilde{\ca{O}}_\lambda^*$ are defined in Definition~\ref{def:structure_sheaf_quot}. Note that $N_{\lambda,\mu}^{\nu}=N(0)_{\lambda,\mu}^{\nu}$ by Theorem~\ref{thm:Intro_structure_constants}.
\end{definition}

Next, we define a Topological Quantum Field Theory (TQFT) valued in the power series ring
$\mathbf{R}=\bb{C}[[q]]$, as considered in~\cite{BryanPandharipande}.
Let $\mathbf{2Cob}$ be the catergory where objects are disjoint union of circles
and morphisms are diffeomorphism classes of cobordisms.
The composition of morphisms is given by the concatenation of cobordisms.
\begin{definition}
	A (1+1)-TQFT with coefficients in $\mathbf{R}$ is a symmetric monoidal functor
	$$F:\mathbf{2Cob}\to \mathbf{Mod_{R}}, $$
	where $\mathbf{Mod_{R}}$ is the category of $\mathbf{R}$-modules. This data is equivalent to:
	\begin{itemize}
		\item[(i)] Let $M$ be an $\mathbf{R}$-module associated to a circle, then $F$ associates $M^{\otimes s}$ to the disjoint union of $s$ circles.
		\item[(ii)] The value of $F$ on the trivial cobordism from a circle to a circle is associated to the identity map on $M$.
		\item[(iii)] The concatenation cobordisms correspond to the composition in $\mathbf{Mod_{R}}$.
	\end{itemize}
\end{definition}
Consider the free $\bb{C}[[q]]$-module $M=K(X)\otimes \bb{C}[[q]]$, with the basis $\{\ca{O}_{\lambda}: \lambda\in\P_{r,k} \}$.
For a genus $g$ cobordism from $s$ circles to $t$ circles,
we define the morphism $M^{\otimes s}\to M^{\otimes t}$ using the matrix given by:
\[
\cO_{\lambda^1}\otimes \cdots \otimes\cO_{\lambda^{s}} \to \sum_{\underline{\nu}}^{} N(g)_{\underline{\lambda}}^{\underline{\nu}}\cdot  \cO_{\nu^1}\otimes \cdots \otimes \cO_{\nu^{t}}.
\]
\begin{theorem}\label{thm:Intro_TQFT}
	The invariants $N(g)_{\underline{\lambda}}^{\underline{\nu}}$ define a (1+1)-TQFT with values in the power series ring $\bb{C}[[q]]$. In particular, we prove the gluing formula
		$$\sum_{\underline{\mu}}^{}N(g_1)_{\underline{\lambda}}^{\underline{\mu}}\cdot N(g_2)_{\underline{\mu}}^{\underline{\nu}}= N(g+t-1)_{\underline{\lambda}}^{\underline{\nu}},$$
		where $g=g_1+g_2$ and the sum runs over all $t$-tuples, $\underline{\mu}= (\mu^1,\cdots,\mu^t)$, of partitions in $\P_{r,k}$. Furthermore, we prove the orthogonality relation $$N(0)_{\lambda}^{\nu} =\sum_{d\ge 0}^{}q^d\chi(\Quot,\widetilde{\ca{O}}_\lambda\cdot \widetilde{\ca{O}}_\nu^*  ) =  \delta_{\lambda,\nu}.$$

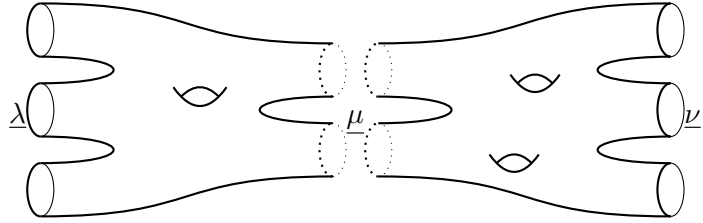
	\[
\underline{\lambda}
\begin{tikzpicture}[baseline=25pt, thick]
\pic[
tqft,rotate=90,
cobordism height=110pt,boundary separation=30pt,
cobordism edge/.style={draw},
between outgoing 2 and 3/.style={dotted},
outgoing upper boundary component 2/.style={dotted},
outgoing lower boundary component 2/.style={dotted},
outgoing upper boundary component 1/.style={dotted},
outgoing lower boundary component 1/.style={dotted},
incoming boundary components=3,
outgoing boundary components=2,offset=.5,
every upper boundary component/.style={draw},
every lower boundary component/.style={draw,thin},
genus=1,
hole 1/.style={rotate=-90,shift={(30pt,90pt)}},
hole 2/.style={rotate=-90,shift={(10pt,70pt)}},
hole 3/.style={rotate=-90,shift={(-30pt,130pt)}},]{};
\end{tikzpicture}
\underline{\mu}
\begin{tikzpicture}[baseline=10pt, thick]
\pic[
tqft,rotate=90,
cobordism height=110pt,boundary separation=30pt,
cobordism edge/.style={draw},
incoming upper boundary component 1/.style={dotted},
incoming lower boundary component 1/.style={dotted},
incoming upper boundary component 2/.style={dotted},
incoming lower boundary component 2/.style={dotted},
incoming boundary 3/.style={dotted},
incoming boundary components=2,
outgoing boundary components=3,offset=-0.5,
every upper boundary component/.style={draw},
every lower boundary component/.style={draw,thin},
genus=2,
hole 1/.style={rotate=-90,shift={(50pt, 50pt)}},
hole 2/.style={rotate=-90,shift={(30pt,80pt)}},
hole 3/.style={rotate=-90,shift={(-30pt,130pt)}},]{}
;
\end{tikzpicture}
\underline{\nu}
\]
\end{theorem}
 We shall see later that the above TQFT matches with the quantum-$K$ TQFT for the Grassmannian. The novelty here is that all the invariants are expressed as (virtual) Euler characteristics over Quot schemes. 
\begin{rem}
	The cohomological Quot scheme invariants also exhibit a TQFT structure with $\bb{C}$-coefficients, called Grassmann/$GL$-Verlinde TQFT in \cite{marian}, which differs from the standard TQFT coming from the quantum cohomology of the Grassmannian. Their proof relies on the fact that the cohomological invariants are enumerative in the large degree regime, which is not applicable in the $K$-theoretic setting.
\end{rem}

\subsection{$K$-theoretic invariants of Quot scheme}
To present the next result, we will introduce notation that is basis-independent for defining invariants over the Quot scheme.

The Quot scheme $\quot_d(C,N,r)$ compactifies $\mathrm{Mor}_d(C,X)$, the space of degree $d$ morphisms from $C$ to $X=\grass$. For any point $p\in C$, there is an evaluation morphims from $\mathrm{Mor}_d(C,\grass)$ to $X$, however this evaluation map does not extend to the Quot scheme. Viewing $\quot_d(C,N,r)$ as a stable quasimap space, see \cite{CKM} for example, the Quot scheme admits an evaluation map $\wev_p$ to a quotient stack $\fr{X}=[M_{r\times N}/\text{GL}_r(\bb{C})]$ satisfying the commutative diagram
\[
\begin{tikzcd}
	\mathrm{Mor}_d(C,X) \arrow[r
	 ,"\ev_p"
	] \arrow[d,hookrightarrow] & X \arrow[d,hookrightarrow] \\
	\quot_d(C,N,r) \arrow[r,"\wev_p"] & \fr{X}
\end{tikzcd}
\]
where $M_{r\times N}$ is the affine space of $r\times N$ matrices, and $X=\grass$ is the quotient space [$M_{r\times N}^{ss}/
\GLr]$ of matrices of full rank. We will explain this `stacky' evaluation map in Section~\ref{sec:K-theory}. Note that the $K$-theory $K^0(\fr{X}) \cong K_{\GLr}^0(M_{r \times N})$ of the stack $\fr{X}$ is isomorphic to the representation ring $R(\GLr)$ of rational $\GLr$-representations. Given a $\GLr$-representation $\mathrm{V}$, we denote its associated $K$-theory class on $\fr{X}$ by the same letter $\mathrm{V} \in K(\fr{X})$. An element $\mathrm{V} \in K(\fr{X})$ restricts to an element $V \in K(X)$.
\begin{definition}
	\label{def:intro-K-theoretic-quot}
	For $\mathrm{V}_1,\dots,\mathrm{V}_n\in K(\fX)$ and points $p_1,\dots,p_n\in C$, we define the $K$-theoretic Quot scheme invariant by
	\[
	\langle
	\mathrm{V}_{1},
	\dots,
	\mathrm{V}_{n}
	\rangle^{\quot}_{g,d}
	:=\chi^{\vir}
	\bigg(
	\quot_d(C,N,r),
	\prod_{i=1}^{n}
	\wev_{p_i}^{*}
	\left(
	\mathrm{V}_{i}
	\right)
	\bigg).
	\]
	We define the corresponding generating series by
	\[\langle\!\langle
	\mathrm{V}_{1},
	\dots,
	\mathrm{V}_{n}
	\rangle\!\rangle^{\quot}_{g} =\sum_{d\ge0}^{}q^d\langle
	\mathrm{V}_{1},
	\dots,
	\mathrm{V}_{n}
	\rangle^{\quot}_{g,d}. \]
\end{definition}
\begin{rem}
	The virtual invariant in Definition~\ref{def:intro-K-theoretic-quot} is independent of the choice of points $p_1, \dots, p_n \in C$ and of the complex structure of the genus $g$ curve $C$. This follows from the virtual Riemann--Roch theorem \cite[Theorem 3.3]{Barbara_Lothar}, which expresses these virtual Euler characteristics as integrals of cohomology classes over the virtual fundamental class (see also \cite{Marian-Oprea}).
\end{rem}
Let $\mathrm{S}\in K(\fr{X})$  be the element corresponding to the dual of the standard representation $(\C^{r})^{\vee}$. Note that $\mathrm{S}$ restricts to the universal subbundle $S$ over $X=\grass$. For $\lambda\in \prk$, we define the `stacky' structure sheaf $\mathrm{O}_\lambda$ in $K(\fr{X})$ by substituting $\ca{S}_p$ with $\mathrm{S}$
in Definition~\ref{def:structure_sheaf_quot}.
In the representation ring $R(\GLr)$, $\mathrm{O}_\lambda$ is given by $\bG^{\lambda}(\bb{C}^r)$,
which is the Grothendieck functor applied to the standard represntation~\cite[\textsection{8}]{Lenart} (see Section~\ref{sec:K-theory} for details). Note that the stacky pullback satisfies $\wev_{p}(\mathrm{S})= \ca{S}_p$. With this notation, the invariants in Definition \ref{defn:intro-TQFT-invariants} equal
\[
N(g)_{\underline{\lambda}}^{\underline{\nu}} =	\langle\!\langle
\mathrm{O}_{\lambda^1},
\dots,
\mathrm{O}_{\lambda^s},\mathrm{O}_{\nu^1}^*,\dots,\mathrm{O}_{\nu^t}^*
\rangle\!\rangle^{\quot}_{g},
\]
where $\mathrm{O}_{\nu}^*=\mathrm{O}_{\nu^*}\cdot \det(\mathrm{S}) \in K(\fr{X})$ is referred to as the \emph{quantized $K$-theoretic dual} in this paper.
\subsection{Quantum reduction map}
Let $\qk(X)$ be the quantum $K$-ring of the Grassmannian $X=\grass$. 
To better understand the quantum $K$-ring, we consider the following:
\begin{definition}
	The quantum reduction map $\kappa: 
	R(\GLr)
	\otimes
	\bb{Z}[[q]]
	\ra
	\qk(X)$ is defined to be the $\bb{Z}[[q]]$-linear extension of the formula
	\begin{equation}
		\label
		{eq:kappa-map}
		\begin{aligned}
			\kappa(\mathrm{V})=
			\sum_{\alpha\in\P_{r,k}}
			\langle\!\langle
			\mathrm{V}
			,
			\mathrm{O}^*_{\alpha}
			\rangle\!\rangle^{\quot}_0
			\ca{O}_{\alpha}
		\end{aligned},
	\end{equation}
	where $\mathrm{V}\in R(\GLr)\cong K(\fr{X})$. 
\end{definition}
Let $\ca{R}^{r,k}$ be the subspace of $R(\GLr)$ spanned by $\{\mathrm{O}_{\lambda}\}_{\lambda\in \P_{r,k}}$. The orthogonality relation in Theorem~\ref{thm:Intro_TQFT} implies that the map $\kappa$ sends $\mathrm{V}$ to its restriction
class $V\in K(X)$ for any representation $\mathrm{V}$
inside the `rectangle box' $\ca{R}^{r,k}$.
When $\mathrm{V}$ is outside of $\ca{R}^{r,k}$, the map $\kappa$
`reduces' it to an element that usually depends on $q$ in the quantum $K$-ring. For example, Proposition~\ref{prop:inverse-det-kappa-image} states that 
\begin{equation*}
\kappa(\det\mathrm{S}^{\vee}) =\sum_{d\ge 0}^{}q^d\det{S}^{\vee}= \frac{\det{S}^{\vee}}{1-q},
\end{equation*}
where $\det\mathrm{S}^{\vee}$ corresponds to the determinant of the standard $\GLr$-representation.
\begin{theorem}\label{thm:intro_kappa_homomorphism}
	The quantum reduction map $\kappa$ is a surjective ring homomorphism.
\end{theorem}

Since the product in the representation ring $R(\GLr)$ is well-understood,
$\kappa$ can be used to analyze and reconstruct the more intricate quantum $K$-ring.
Specifically, a presentation of $\qk(X)$
can be formulated by identifying relations in the kernel of $\kappa$.
To illustrate this point, we provide an independent proof of the relations in the quantum $K$-Whitney presentation for $\qk(X)$,
proposed in the physics literature \cite{GMSZ1} and proved mathematically in \cite[Theorem 1.1]{GMSZ2}:
For any vector bundle $E\to X$ of rank $e$, we define the polynomial $$\wedge_y E= 1+yE+y^2\wedge^2E+\cdots +y^{e}\wedge^eE.$$
\begin{cor}
	Let $S$ and ${Q}$ be the universal subbundle and quotient bundle over $X=\grass$. In $\qk(X)$, 
	\begin{align}\label{eq:into_QK_Whitney}
	\wedge_y S\bullet \wedge_y  Q=(1+y)^N-\frac{q}{1-q}y^{k}(\wedge_y S - 1)\bullet \det  Q.
	\end{align}
\end{cor}
The quantum $K$-Whitney presentation for $\qk(X)$ is given by the algebra generated by exterior powers of $S$ and $ Q$, modulo the relations given by \eqref{eq:into_QK_Whitney}. These relations are equivalent to those in the `Coulomb branch' presentation in \cite{GMSZ1} (see Corollary~\ref{cor:QK_Coulomb})
$$
\wedge_y S\bullet \wedge_y  Q = (1+y)^N - qy^k(\wedge_yS-1)\cdot \det  Q.
$$
We re-establish \eqref{eq:into_QK_Whitney} using the homomorphism property of $\kappa$ from
Theorem~\ref{thm:intro_kappa_homomorphism} and the vanishing results for the Euler
characteristics over Quot scheme of $\bb{P}^1$ mentioned in the next subsection.

\subsection{Vanishing of Euler characteristics}
The proofs of Theorems~\ref{thm:Intro_structure_constants}, \ref{thm:Intro_TQFT} and \ref{thm:intro_kappa_homomorphism}
rely on three main ingredients: the vanishing of Euler characteristics,
degeneration formulas for the Quot schemes,
and wall-crossing formulas.
The latter two will be explained in the next two subsections and proved in Sections 5-7.
The first is proved in~\cite{SinhaZhang2} using equivariant torus localization (see also the first author's Ph.D. thesis~\cite{Sinha_thesis}).
We recall the statement below.

\begin{theorem}[{\cite{SinhaZhang2}}]
	\label
	{thm:vanishing}
	Let $\lambda=(\lambda_{1},\dots,\lambda_{r})$ be a nonempty partition.
	We have 
	\begin{align*}
	\langle
	\bS^\lambda(\mathrm{S})\rangle^{\quot}_{0,d}  = \chi(\Quot, \bS^\lambda(\cS_p))=0.
	\end{align*} 
	if one of the following two conditions holds
	\begin{enumerate}[\normalfont(i)]
		\item $d\geq\lambda_{1} -(N-r)$;
		\item $d\geq \lambda_{1}-2(N-r)$ with additional conditions $d>0$,$r\ne N$, and $\lambda_r>0$;
	\end{enumerate}
\end{theorem}

When $d=0$, the Quot scheme is isomorphic to $\grass$ and the vanishing in part (i) agrees with
the Borel--Weil--Bott theorem (c.f.~\cite{Kapranov2}).

\subsection{Wall-crossings formula}
To compare the $K$-theoretic Quot scheme invariants with (Givental-Lee's) quantum $K$-invariants of the Grassmannian, we apply two sequences of wall-crossings.
First, we apply the $K$-theoretic quasimap wall-crossing formula in~\cite{ZZ1} to show that quantum $K$-invariants are equal to the $K$-theoretic $(\epsilon=0+)$-stable quasimap invariants.
Second, we recall the light-to-heavy wall-crossing formula in~\cite{RZ2} which implies that $K$-theoretic Quot scheme invariants equal $K$-theoretic $(0+)$-stable \emph{graph} quasimap invariants. The
$K$-theoretic $(0+)$-stable graph quasimap invariants and
$K$-theoretic $(\epsilon=0+)$-stable quasimap invariants
coincide when the number of markings is less than or equal to three.

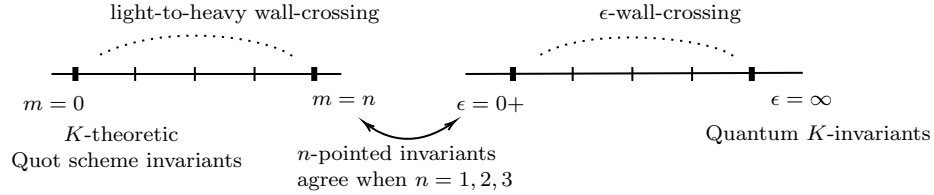
\begin{center}	
	\tikzset{every picture/.style={line width=0.75pt}} 
	
	\begin{tikzpicture}[x=0.75pt,y=0.75pt,yscale=-0.5,xscale=0.5]
		
		\draw [line width=0.75]    (56.46,171.77) -- (347.33,171.77) ;
		\draw [line width=2.25]    (80.46,164.75) -- (80.46,178.8) ;
		\draw [line width=0.75]    (140.39,164.75) -- (140.39,178.8) ;
		\draw [line width=0.75]    (200.32,164.75) -- (200.32,178.8) ;
		\draw [line width=0.75]    (260.26,164.75) -- (260.26,178.8) ;
		\draw [line width=2.25]    (320.19,164.75) -- (320.19,178.8) ;
		\draw  [dash pattern={on 0.84pt off 2.51pt}]  (298,153.61) .. controls (272,134.95) and (170.35,119.24) .. (103.4,153.35) ;
		\draw [line width=0.75]    (472,171.77) -- (810.51,170.59) ;
		\draw [line width=2.25]    (519.64,164.75) -- (519.64,178.8) ;
		\draw [line width=0.75]    (579.57,164.75) -- (579.57,178.8) ;
		\draw [line width=0.75]    (639.5,164.75) -- (639.5,178.8) ;
		\draw [line width=0.75]    (699.44,164.75) -- (699.44,178.8) ;
		\draw [line width=2.25]    (759.37,164.75) -- (759.37,178.8) ;
		\draw  [dash pattern={on 0.84pt off 2.51pt}]  (740.67,153.57) .. controls (714.67,134.9) and (594,123.61) .. (546.07,153.3) ;
		\draw    (371.72,221.4) .. controls (390.46,236.23) and (422.22,247.81) .. (460.5,219.81) ;
		\draw [shift={(461.67,218.95)}, rotate = 143.13] [color={rgb, 255:red, 0; green, 0; blue, 0 }  ][line width=0.75]    (10.93,-3.29) .. controls (6.95,-1.4) and (3.31,-0.3) .. (0,0) .. controls (3.31,0.3) and (6.95,1.4) .. (10.93,3.29)   ;
		\draw [shift={(370,220)}, rotate = 39.91] [color={rgb, 255:red, 0; green, 0; blue, 0 }  ][line width=0.75]    (10.93,-3.29) .. controls (6.95,-1.4) and (3.31,-0.3) .. (0,0) .. controls (3.31,0.3) and (6.95,1.4) .. (10.93,3.29)   ;

		\draw (214.53,156.84) node [anchor=north west][inner sep=0.75pt]  [rotate=-179.34] [align=left] {$ $};
		\draw (112.49,98.67) node [anchor=north west][inner sep=0.75pt]  [font=\scriptsize]   [align=left] {light-to-heavy wall-crossing};
		\draw (25,190.7) node [anchor=north west][inner sep=0.75pt]  [font=\scriptsize] [align=left] {$\displaystyle m=0$};
		\draw (315,190.67) node [anchor=north west][inner sep=0.75pt]  [font=\scriptsize] [align=left] {$\displaystyle m=n$};
		\draw (642.76,182.17) node [anchor=north west][inner sep=0.75pt]   [align=left] {$ $};
		\draw (603.33,98.7) node [anchor=north west][inner sep=0.75pt]   [font=\scriptsize]  [align=left] {$\displaystyle \epsilon $-wall-crossing};
		\draw (460,190.67) node [anchor=north west][inner sep=0.75pt]  [font=\scriptsize] [align=left] {$\displaystyle \epsilon =0+$};
		\draw (775.33,190.67) node [anchor=north west][inner sep=0.75pt]  [font=\scriptsize] [align=left] {$\displaystyle \epsilon =\infty $};
		\draw (300,240) node [anchor=north west][inner sep=0.75pt]  [font=\scriptsize] [align=left] {$n$-pointed invariants \\agree when $n=1,2,3$};
		\draw (13,220.83) node [anchor=north west][inner sep=0.75pt]   [font=\scriptsize]  [align=left] { $\displaystyle \ \ \ \ \ \ \ K$-theoretic \\Quot scheme invariants};
		\draw (710,220.8) node [anchor=north west][inner sep=0.75pt]    [font=\scriptsize] [align=left] {Quantum $\displaystyle K$-invariants};
	\end{tikzpicture}
\end{center}

More precisely, we prove the following identity for the stable maps invariants from \eqref{eq:intro-quantum-k} (see Theorem~\ref{prop:QK=quot} and Corollary~\ref{cor:identify-2-point} for the equivariant version):
\begin{theorem}
	\label
	{thm:intro-Wall_crosing}
	 Let $\mathrm{V}_1,\mathrm{V}_2,\mathrm{V}_3\in \ca{R}^{r,k}\subset K(\fr{X})$, where $k=N-r>0$.
	Denote $V_1,
	V_2,
	V_3$ as their restrictions to $K(X)$.
	Then 
	\[
		\langle\!\langle
	\mathrm{V}_1,
	\mathrm{V}_2,
	\mathrm{V}_3
	\rangle\!\rangle^{\quot}_{0}
	=
	\langle\!\langle
V_1,
V_2,
V_3
	\rangle\!\rangle^{}_{0,3},
	\]
where the right-hand side is the generating series of the quantum $K$-invariants~\eqref{eq:intro-quantum-k}.
The same holds for the 2-pointed and 1-pointed invariants.
\end{theorem}

Recall the definition of the
correlation functions $F_{\lambda,\mu,\nu}$ and $F_{\lambda,\mu}$
of quantum $K$-invariants (see~\eqref{eq:def_QK-invariats_structure_sheaves}). Theorem~\ref{thm:intro-Wall_crosing} shows that:
\begin{cor}\label{thm:3-point_invariants_QK_equals_Quot}
	For any partitions $\lambda,\mu,\nu\in \P_{r,k}$, we have 
	$$F_{\lambda,\mu,\nu} = N(0)_{\lambda,\mu,\nu}\quad \text{and}\quad F_{\lambda,\mu} = N(0)_{\lambda,\mu}.$$
\end{cor}
The quantum $K$-ring
equipped with the quantized pairing $F_{\lambda,\mu}$ is a $\bb{Z}[[q]]$-valued Frobenius
algebra. The above corollary implies that the TQFT structure in Theorem~\ref{thm:Intro_TQFT} coincides with the one determined by
the quantum $K$-Frobenius algebra $\qk(X)$.

 \begin{rem}
 Theorem~\ref{thm:intro-Wall_crosing} implies that the stable map invariants $ 	\langle\!\langle \mathcal{O}_{\lambda}, \mathcal{O}_{\mu}, \mathcal{O}_{\nu}^* \rangle\!\rangle_{0,3}$
 	coincide with the structure constants 
 	$
 	N_{\lambda, \mu}^{\nu} = \langle\!\langle \mathrm{O}_{\lambda}, \mathrm{O}_{\mu}, \mathrm{O}_{\nu}^* \rangle\!\rangle_{0}^{\quot}
 	$
 	when $\mathrm{O}_{\nu}^* \in \mathcal{R}^{r,k}$. However, this equality may fail when $\mathrm{O}_{\nu}^*$ lies outside $\mathcal{R}^{r,k}$, in particular when the partition $\nu$ has strictly fewer than $r$ parts.
	The discrepancy is clarified in Remark~\ref{rem:Quantum_dual}, which shows
	that the quantized dual $\mathrm{O}_{\nu}^*\in K(\fr{X})$ loses information
	when restricted to $\ca{O}_{\nu}^*$ over $X=\grass$.
\end{rem}

\subsection{Degeneration formula}
In algebraic geometry, the degeneration formulas for Gromov--Witten invariants
and Donaldson--Thomas theory were established in~\cite{Jun1,Jun2}
and~\cite{MPT, Li-Wu}, respectively.
Subsequent work extended these formulas to $K$-theoretic Donaldson--Thomas invariants
\cite{Qu} and to $K$-theoretic quasimap invariants of Nakajima quiver varieties~\cite{Okounkov}.
In Section~\ref{sec:gluing}, we prove the following two degeneration formulas for virtual $K$-theoretic invariants of the Quot scheme for curves of all genera:

\begin{theorem}\label{thm:intro_degeneration_formula}
Let $F^{\alpha,\beta}$ be the inverse of the quantized pairing matrix define in~\eqref{eq:def_QK-invariats_structure_sheaves}.
	\begin{enumerate}[(i)]
\item Let $\mathrm{V},\mathrm{W}\in K(\fr{X})$. Then
	\begin{equation}
		\label{eq:degeneration-1}
	\langle\!\langle
	\mathrm{V},
	\mathrm{W} \rangle\!\rangle^{\quot}_{g}=\sum_{\alpha,\beta\in\P_{r,k}} \langle\!\langle
	\mathrm{V},
	\mathrm{O}_{\alpha}
	\rangle\!\rangle^{\quot}_{g_1}\cdot F^{\alpha,\beta}\cdot \langle\!\langle
	\mathrm{O}_{\beta},	\mathrm{W}
	\rangle\!\rangle^{\quot}_{g_2},
	\end{equation}
	where $\alpha$ and $\beta$ range over all partitions in $\P_{r,k}$, and $g_1+g_2=g$. 
	\item For any $\mathrm{V}\in K(\fr{X})$ and $g>0$:
	\begin{equation}
		\label{eq:degeneration-2}
	\langle\!\langle
	\mathrm{V} \rangle\!\rangle^{\quot}_{g}=\sum_{\alpha,\beta\in\P_{r,k}} \langle\!\langle
	\mathrm{V},
	\mathrm{O}_{\alpha}, \mathrm{O}_{\beta}
	\rangle\!\rangle^{\quot}_{g-1}\cdot F^{\alpha,\beta}. 
	\end{equation}
\end{enumerate}
\end{theorem}
When $\mathrm{V},
\mathrm{W}\in \ca{R}^{r,2k}$ and $g=0$, the degeneration formulas follow from the WDVV equation for
quantum $K$-theory proved in \cite{WDVV} and \cite{Lee}.
Our proof of the above degeneration formulas leverages the good degeneration of Quot schemes developed in~\cite{Li-Wu}
and is modeled on that in~\cite[\textsection{4}]{Qu}.

\begin{rem}
	Let $T\cong(\C^{*})^{N}\subset \mathrm{GL}_{N}(\C)$ be the maximal torus of diagonal
	matrices acting on $X=\grass$ and $\quot_d(C,\bb{}N,r)$.
	The vanishing results in Theorem~\ref{thm:vanishing},
	the wall-crossing result in Theorem~\ref{thm:intro-Wall_crosing},
	and the degeneration formulas in Theorem~\ref{thm:intro_degeneration_formula}
	all hold in the torus-equivariant setting. However, Theorems~\ref{thm:Intro_structure_constants}, \ref{thm:Intro_TQFT} and~\ref{thm:intro_kappa_homomorphism} do not hold as stated in the equivariant setting.
\end{rem}

\subsection{Future directions}
In this paper, we explore the quantum $K$-theory of the Grassmannian
via Quot schemes.
The Quot scheme on $\bP^1$ can be thought of as a \emph{quasimap space}
that parametrizes maps $\bP^1\rightarrow \fr{X}$ to the quotient
stack $\fr{X}$ such that the generic point of $\bP^1$
is mapped to the Grassmannian $X\subset \fr{X}$.
In general, let $G$ be a reductive group acting on an affine variety $W$
and let $\theta$ be a character of $G$.
Replacing $X$ and $\fr{X}$ with the GIT quotient $W/\!/_\theta G$
and the quotient stack $[W/G]$ respectively, we define the quasimap space
to $W/\!/_\theta G$.

Under some mild conditions (see p.341 of~\cite{KOUY}), the quasimap space has a canonical perfect obstruction
theory, as detailed in~\cite[Theorem 7.2.2]{CKM} and~\cite[\textsection{10.7}]{KOUY}. This allows us to define virtual cohomological invariants, useful for studying the quantum cohomology and Gromov--Witten--Ruan invariants of the GIT quotient.

This approach has been applied to target spaces such as isotropic Grassmannians and
partial flag manifolds (see \cite{CF1, CF2, CF3, Chen1, Sinha}).
We hope to extend these results to $K$-theory, particularly to
prove the degeneration formula for the $K$-theoretic quasimap invariants
of a general GIT quotient and to use the quantum reduction map to derive presentations of
its quantum $K$-ring.

\subsection{Plan of the paper}
In Section~\ref{sec:K-theory}, we cover preliminaries $K$-theory,
quantum $K$-theory, and virtual structure sheaves. In Section~\ref{sec:TQFT}, we prove Theorem~\ref{thm:Intro_structure_constants} and
Theorem~\ref{thm:Intro_TQFT},
assuming the vanishing Theorem~\ref{thm:vanishing}, the wall-crossing result Corollary~\ref{thm:3-point_invariants_QK_equals_Quot},
and the degeneration formulas
from Theorem~\ref{thm:intro_degeneration_formula}.
In Section \ref{sec:quantum-reduction}, we introduce the quantum reduction
map $\kappa$, prove Theorem~\ref{thm:kappa-homo}, and use it to derive the Whitney presentation for $\qk(X)$.
In Section~\ref{sec:Quasimap wall-crossing} and \ref{sec:Light-to-heavy wall crossing},
we prove the wall-crossing formulas used in Sections~\ref{sec:TQFT} and~\ref{sec:gluing}. Two degeneration formulas are proven in Section~\ref{sec:gluing}, which is crucially used in Section~\ref{sec:TQFT} to show the TQFT structure of the $K$-theoretic Quot scheme invariants.

\subsection{Acknowledgment}
Both authors are grateful to Leonardo C. Mihalcea for sharing insights on
the $K$-theoretic divisor equation and his related conjecture with Buch~\cite[Conjecture 4.3]{GMSXZZ}.
This paper, together with the adjoining~\cite{SinhaZhang2},
was inspired by our efforts to verify the $K$-theoretic divisor equation via Quot schemes.
The first author thanks Alina Marian, Dragos Oprea, and Woonam Lim for valuable conversations about the Quot schemes of curve.
The second author would like to thank Wei Gu and Du Pei for explaining the physics connected to this work,
and thanks Reuven Hodges and Yongbin Ruan for helpful discussions.

\section{Preliminaries}
\label
{sec:K-theory}

\subsection{Algebraic representations of $\GLr$}
Recall that a \emph{partition} (with $r$ nonnegative parts) is a weakly decreasing sequence of nonnegative integers
$\lambda=(\lambda_{1},\dots,\lambda_{r})$, where $\lambda_{1}\geq \lambda_{2}\geq \dots\geq \lambda_{r}\geq 0$. 	
Let
\[
\P_{r,k}
=\{
\lambda=(\lambda_{1},\dots,\lambda_{r})
\mid
\lambda_{1}\leq k
\}
\]
be the set of partitions contained in the rectangle $(k)^{r}=(k,k,\dots,k)$ with $r$ rows and $k$ columns. Let $\lambda^\dagger$ denote the conjugate of the partition $\lambda$, and let $\ell(\lambda^{\dagger})=\lambda_1$ denote the number of parts of $\lambda^\dagger$. For any partition $\lambda = (\lambda_1, \dots, \lambda_r)$, the corresponding \emph{Schur polynomial} can be defined using the (second) Jacobi--Trudi formula:
\begin{align}\label{eq:Jacobi-Trudi}
s_{\lambda}(x_1, \dots, x_r) := \det\left( e_{\lambda^\dagger_i + j - i}(x_1, \dots, x_r) \right)_{i,j=1}^{\ell(\lambda^\dagger)},
\end{align}
where $e_m(x_1, \dots, x_r)$ denotes the $m^{\text{th}}$ elementary symmetric polynomial.

 Let $\bS^{\lambda}$ be the \emph{Schur functor} associated to $\lambda$ (see~\cite[\textsection{6}]{Fulton-Harris}) and let $\bb{C}^r$ be the standard representation of $\GLr$.
 Then $\bS^{\lambda}(\C^{r})$ is the unique irreducible (polynomial) representation of $\GLr$ of highest weight $\lambda$.
 For example, if $\lambda=(a)$, then $\bS^{(a)}(\C^{r})=\Sym^{a}(\C^{r})$; if $\lambda=(1)^{b}:=(1,\dots,1)$ with 1 repeated $b$ times, then $\bS^{(1)^{b}}(\C^{r})=\wedge^{b}(\C^{r})$. Let $R^0(\GLr)$ denote the ring of polynomial representations of $\GLr$. Note that, for any $g\in\GLr$, the trace of $g$ on $\bS^{\lambda}(\C^{r})$ is given by
 \[
 \chi_{\bS^{\lambda}(\bb{C}^{r})}(g)
 =
 s_{\lambda}(x_{1},\dots,x_{r}),
 \]
 where $x_{1},\dots,x_{k}$ are the eigenvalues of $g$ and $s_{\lambda}$ is the Schur polynomial.
 This gives an isomorphism $\chi$ between $R^0(\GLr)$ and the ring of symmetric polynomials in $r$ variables. 
 
 Let $R(\GLr)$ denote the ring of algebraic representations of $\GLr$. As a ring, $R(\GLr)$ is obtained from $R^0(\GLr)$ by inverting the determinant representation $\det(\bb{C}^r)$. The trace map $\chi$ extends to the isomorphism between $R(\GLr)$ and the ring of symmetric Laurent polynomials in the $r$ variables $x_1,\dots,x_r$.

Let $\mathrm{S}:=\big(\bb{C}^{r}\big)^{\vee}$ denote the dual of the standard representation. Taking the dual of representations gives an automorphism of the representation ring $R(\GLr)$. At the level of Laurent polynomials, it is given by mapping $x_i\to x_i^{-1}$. In particular, the representation ring $R(\GLr)$ is also generated by $\{1,\mathrm{S},\wedge^2\mathrm{S},\dots,\det(\mathrm{S}), \det(\mathrm{S})^{\vee}\}$ over $\bb{C}$. Let $\ca{R}^{r}$ denote subring $R(\GLr)$ generated by $\{1,\mathrm{S},\wedge^2\mathrm{S},\dots,\det(\mathrm{S})\}$, that is, $\ca{R}^{r}$ corresponds to the subring consisting of symmetric polynomials in $x_1^{-1},\dots,x_r^{-1}$ inside the ring of symmetric Laurent polynomials.

	\begin{definition}
		\label{def:subspace-rep-ring}
			Given a positive integer $\ell$, let $\ca{R}^{r,\ell}$ be the linear subspace of $\ca{R}^r$ spanned by polynomials in $$\{\mathrm{S},\wedge^2\mathrm{S},\dots,\det(\mathrm{S})\}$$ of degree at most $\ell$, providing a filteration for $\ca{R}^r$. Note that each element $\wedge^i\mathrm{S}\subset \ca{R}^{r,1}$. The Jacobi-Trudi formula \eqref{eq:Jacobi-Trudi} implies that  $\ca{R}^{r,\ell}$ is spanned by
		$\{
		\bS^{\lambda}(\mathrm{S}) 
		\mid \lambda\in \P_{r,\ell}
		\}$
		. 
	\end{definition}

	\subsection{Basic notation in $K$-theory}
	Let $\fr{X}$ be a Deligne--Mumford (DM) stack with an action of an algebraic torus $T$. 
	The group $\kG(\fr{X})$ (resp. $\kg(\fr{X})$) denotes the Grothendieck group of $T$-equivariant coherent sheaves (resp. $T$-equivariant locally free sheaves) on $\fr{X}$. Given a (equivariant) coherent sheaf (or vector bundle) $E$, we denote by $[E]$ or
	simply $E$ its associated $K$-theory class (see \cite{Kiem-Savvas} for properties of the $K$-theory of DM stacks and the virtual structure sheaf, which we briefly describe below.). The group $\kG(\fr{X})$ has a
	$\kg(\fr{X})$-module structure defined by the tensor product
	\[
\kg(\fr{X})\otimes \kG(\fr{X})\rightarrow \kG \fr{X},\quad \big([E],[F]\big)\mapsto [E]\cdot [F]:=[E\otimes_{\ca{O}_\fr{X}}F].
	\]
	We denote by $K_0(\fr{X})$ and $K^0(\fr{X})$ the non-equivariant $K$-groups of
	coherent sheaves and locally free sheaves, respectively.
	
	If $X$ is nonsingular scheme, the map $K^0(X)\rightarrow K_0(X)$ sending a vector bundle to its sheaf of sections is an isomorphism. In this case, we denote the non-equivariant $K$-group of $X$ by $K(X)$.
	
	Let $\mathrm{pt}=\spec\bb{C}$.
	If we write $T= (\C^{*})^{N}$, then $$\Gamma:=\kg(\mathrm{pt}) =K^0(BT)\cong \bb{Z}[\alpha_{1}^{\pm},\dots, \alpha_{N}^{\pm}],$$ where $\alpha_{i}$ is the character of $T$ given by the projection to the $i$-th factor. Given a character $\alpha$ of $T$, let $\C_{\alpha}$ denote the corresponding 1-dimensional representation of $T$.

	Let $F$ be a $T$-equivariant vector bundle on $\fr{X}$.
	We define the \emph{$K$-theoretic Euler
		class} of $F$ by
	\[
	\wedge_{-1}^T(F^\vee): = \sum_{i\geq 0}(-1)^i \wedge^i 
	F^\vee \in \kg(\fr{X}),
	\]
	where $F^{\vee}$ denotes the dual of $F$ and 
	$\wedge^i F^\vee$ denotes the $i$-th exterior power of $F^\vee$.

	Given a $T$-equivariant morphism $f:\fr{X}\rightarrow \fr{Y}$,
	we have the pullback map
	$f^*: \kg(\fr{Y}) \rightarrow \kg(\fr{X})$.
	If $f$ is proper, we have the proper pushforward $f_*:\kG(\fr{X}) \rightarrow \kG(\fr{Y})$ defined by
	\[
	[F] \mapsto \sum_n(-1)^n[R^nf_*F].
	\]
	Given a class $F\in \kG(\fr{X})$ on a proper DM stack $\fr{X}$, we define its $T$-equivariant Euler characteristic
	as the proper pushforward to the point
	\[
	\chi^{T}(\fr{X},F): = \sum_{i \geq0} \text{char}_{T} \, H^i(\fr{X}, F) \in \Gamma.
	\]
	where $\mathrm{char}_{T}$ denotes taking the character of a $T$-module.
	Furthermore, if $\fr{X}$ has a $T$-equivariant virtual struture sheaf
	$\ovir_\fr{X}\in \kG(\fr{X})$ (see Section~\ref{sec:POT}), the $T$-equivariant
	\emph{virtual} Euler characteristic (c.f. \cite{Lee} for instance) by
	\[
	\chi^{\mathrm{vir}, T}(\fr{X}, F)
	:=\chi^T(\fr{X},\ovir_\fr{X}\cdot F)	.
	\]

	\subsection{$K$-theory of the Grassmannian}
	We fix two positive integers $r$ and $N$ such that $N\ge r$.
	Throughout the paper, we set $k=N-r$ and
	$$X=\grass=\{V\subset \C^{N}\mid \dim_{\C} V=r\}$$
	to be the Grassmann variety of $r$-planes in $\C^{N}$. The $K$-theory of the Grassmannian $K(X)$ is linearly spanned by the \emph{Schubert structure sheaves}:
The \emph{Schubert variety} associated to a partition $\lambda \in \P_{r,N-r}$, relative to the Borel subgroup of $\mathrm{GL}_N$ stabilizing a complete flag
\[
0 = F_0 \subset F_1 \subset \dots \subset F_N = \mathbb{C}^N,
\]
is defined by
	\begin{equation}\label{eq:Schubert_cyle}
		X_{\lambda}
		=
		\{
		V\in X
		\mid
		\dim (V\cap F_{k+i-\lambda_{i}})
		\geq i,
		\quad \forall\,
		1\leq i\leq r
		\}.
	\end{equation}
	The codimension of $X_{\lambda}$ equals the weight $|\lambda|=\sum \lambda_{i}$ of $\lambda$.
	The $K$-theory class $\O_{\lambda}:=[\O_{X_{\lambda}}]$ of the structure sheaf of $X_{\lambda}$ is referred to as the Schubert structure sheaf.
	The set $\{\O_{\lambda}\}_{\lambda\in \P_{r,k}}$
	is a $\Z$-basis of $K(X)$.

Here we briefly describe the Schur bundles on $X$ that give a different basis for $K(X)$. We consider the following standard GIT presentation of $X$:
\[
X=[M_{r\times N}^{ss}/\text{GL}_r(\bb{C})],
\]
where $M_{r\times N}$ is the affine space of $r$ by $N$ complex matrices and $M_{r\times N}^{ss}\subset M_{r\times N}$ is the open subset of matrices of full rank. For any irreducible $\GLr$ representation $\bS^{\lambda}(\mathrm{S})$, where $\mathrm{S}=(\bb{C}^{r})^\vee$ is the dual of standard representation, we denote have the following vector bundle on $X$ as the Schur bundle corresponding to a partition $\lambda$
\[
\bS^{\lambda}(S):=M_{r\times N}^{ss}\times_{\GLr} \bS^{\lambda}(\mathrm{S}).
\]
The set $\{\bS^{\lambda}(S)\}_{\lambda\in\P_{r,k}}$ is a $\bb{Z}$-basis of $K(X)$.	
	\begin{rem}
	Let $T\cong(\C^{*})^{N}\subset \mathrm{GL}_{N}(\C)$ be the maximal torus of diagonal matrices.
	Then $T$ acts on $X$.
	There is a tautological subbundle $S\subset \O^{\oplus N}_{X}=X\times \C^{N}$ of rank $r$ on $X$,
	and it has a natural $T$-linearization. Therefore, Schur bundles define a class in the equivariant $K$-theory $K_T(X)$.
\end{rem}

\subsection{Grothendieck functor}
Let $G_{\lambda}$ be the stable Grothendieck polynomial for the partition $\lambda$ (see~\cite{Buch}).
It has the following Weyl-type bialternant formula:
\begin{equation*}
	G_{\lambda}(x_{1},\dots, x_{r})=
	\frac
	{\det
		(x_{j}^{\lambda_{i}+r-i}(1-x_{j})^{i-1}
		)_{1\leq i,j\leq r}
	}
	{
		\det
		( x_{j}^{r-i}
		)_{1\leq i,j\leq r}
	}
	.
\end{equation*}
By~\cite[Theorem 2.2]{Lenart}, we have the following expansion
\begin{equation}
	\label
	{eq:schur-expansion}
	G_{\lambda}
	(x_{1},\dots,x_{r})
	=
	\sum_{\lambda\subseteq \mu \subseteq \hat{\lambda}}
	(-1)^{|\mu|-|\lambda|}
	g_{\lambda\mu}
	s_{\mu}
	(x_{1},\dots,x_{r}),
\end{equation}
where $\hat{\lambda}$ is the unique maximal partition with $r$ rows obtained from $\lambda$ by adding at most $i-1$
boxes to its $i$th row for $2\leq i\leq r$,
and the coefficients $g_{\lambda,\mu}$ are nonnegative and count certain skew tableaux.

According to~\cite[\textsection{8}]{Lenart},
there is a well-defined virtual $\GLr$-representation $\bG^{\lambda}(\bb{C}^r)$ whose character satisfies
\begin{equation}\label{eq:Grothendieck_poly}
	\chi_{\bG^{\lambda}(V)}(g)=
	G_{\lambda}(1-x_{1}^{-1},\dots,1- x_{r}^{-1})
	,
\end{equation}
where $x_{1},\dots,x_{r}$ are the eigenvalues of $g$ on $V$.

In the $K$-theory of $X = \grass$, the structure sheaves $\mathcal{O}_\lambda$ of Schubert varieties (as in \eqref{eq:Schubert_cyle}) can be expressed in terms of Schur functors as
	\begin{equation*}
		\mathcal{O}_\lambda = \sum_{\nu \in \mathcal{P}_{r,k}} \mathrm{A}_\lambda^\nu \ \mathbb{S}^\nu(S),
	\end{equation*}
	for any $\lambda \in \mathcal{P}_{r,k}$. The coefficients $\mathrm{A}_\lambda^\nu$ are computed by expressing the character \eqref{eq:Grothendieck_poly}, the \emph{stable Grothendieck polynomial}, in terms of Schur polynomials.

	Let $\lambda$ be a partition with at most $r$ parts, then the Grothendieck functor is given by
	$$\bG^{\lambda}(\C^{r}) := \sum_{\nu\in\P_{r,\lambda_{1}}}\mathrm{A}_{\lambda}^{\nu}\bb{S}^{\nu}(\mathrm{S}),$$
	where $\mathrm{S}=(\C^{r})^\vee$ is the dual representation.
	The coefficients $\mathrm{A}_{\lambda}^{\nu}$ can be calculated using \eqref{eq:schur-expansion} and expressing $s_{\mu}(1-x_{1}^{-1},\dots,1-x_{r}^{-1})$ as a linear combination of $s_{\nu}(x_{1}^{-1},\dots,x_{r}^{-1})$. 
	
	The description of $g_{\lambda,\mu}$ in~\eqref{eq:schur-expansion}
	and the fact that $s_{\lambda}(1-x_{1}^{-1},\dots,1-x_{r}^{-1})$ can be written as a linear combination of $s_{\mu}(x_{1}^{-1},\dots,x_{r}^{-1})$ with $\mu\subseteq\lambda$ (see, e.g.,~\cite[\textsection{3}, Example 10]{Macdonald}) implies that $	\ca{R}^{r,k}$ in Definition~\ref{def:subspace-rep-ring} is spanned by $\{\bG^{\lambda}(\C^{r}) : \lambda\in P_{r,k}\}$. This eludes to the fact that the Schubert structure sheaves $\{\ca{O}_\lambda: \lambda\in \P_{r,k} \}$ give a linear basis for the $K$-theory of $\grass$.

	\subsection{Quantum $K$-theory}
	There are different variants of quantum $K$-invariants (see, for example, ~\cite{Lee,RZ1,RZ2,Okounkov}).
	In this subsection, we recall the definition of genus-zero quantum $K$-invariants of Grassmannians in the sense of Givental--Lee~\cite{Lee,WDVV,Buch-Mihalcea}.
	
	Let $X=\grass$ be a Grassmannian. For $d\in H_{2}(X)\cong \Z$, let $\ms$ be the moduli stack of degree-$d$, genus zero, $n$-pointed \emph{stable maps} into $X$.
	This is a smooth proper DM stack equipped with \emph{evaluation maps}
	\[
	\ev_{i}:\ms\rightarrow X
	\]
	for $1\leq i \leq n$.
	\begin{definition}
		Let $E_{1},\dots,E_{n}\in \kg(X)$.
		We define the genus zero, $n$-pointed, $T$-equivariant quantum $K$-invariant with the insertions $E_{i}$'s by
		\[
		\langle
		E_{1},
		\dots,
		E_{n}
		\rangle_{0,n,d}^{T}
		:=\chi^{T}
		\bigg(
		\ms,
		\prod_{i=1}^{n}\ev_{i}^{*}(
		E_{i})
		\bigg).
		\]
		The corresponding correlation function is defined by
		\[
		\langle\!\langle
		E_{1}
		,
		\dots
		,
		E_{n}
		\rangle\!\rangle^{T}_{0,n}
		:=
		\sum_{d\geq 0} q^{d}
		\langle
		E_{1},
		\dots,
		E_{n}
		\rangle_{0,n,d}^{T},
		\]
		where $q$ is a formal variable called the Novikov variable.
		The non-equivariant quantum $K$-invariant and its correlation function are denoted by
		$\langle
		E_{1}
		,
		\dots
		,
		E_{n}
		\rangle^{}_{0,n}$
		and $\langle\!\langle
		E_{1}
		,
		\dots
		,
		E_{n}
		\rangle\!\rangle^{}_{0,n}$,
		respectively.
			\end{definition}
	
We assume $k=N-r>0$. Let $\{\ca{O}_{\lambda}\}_{\lambda\in\P_{r,k}}$ be a basis for $K(X)$. We recall the definition of quantum $K$-ring~\cite{WDVV,Lee,Buch-Mihalcea}. In general, we can choose any basis $\{\phi_{\lambda}\}_{\lambda\in \P_{r,k}}$ of $\kg(X)$ and define $T$-equivariant quantum K-ring.

Recall the following generating series of the 3-pointed Givental--Lee's invariants
\begin{equation*}
F_{\alpha,\beta,\gamma}:
=
\langle\!\langle
\ca{O}_{\alpha},
\ca{O}_{\beta},
\ca{O}_{\gamma}
\rangle\!\rangle^{}_{0,3}
=
\sum_{d\geq0}
q^{d}
\langle
\ca{O}_{\alpha},
\ca{O}_{\beta},
\ca{O}_{\gamma}
\rangle_{0,3,d}^{}
\end{equation*}
for partitions $\alpha,\beta,\gamma\in \P_{r,k}$. Similarly we define $F_{\alpha,\beta}=\langle\!\langle
\ca{O}_{\alpha},
\ca{O}_{\beta}
\rangle\!\rangle^{}_{0,2}$.
Let $(F^{\alpha,\beta})$ be the matrix inverse to the quantized pairing matrix
$(F_{\alpha,\beta})$.

Given three partitions $\lambda,\mu,\nu \in \P_{r,k}$, we define the~\emph{structure constant} by
\begin{equation*}
N_{\lambda,\mu}^{\nu}
:=
\sum_{\alpha\in\P_{r,k}}
F^{\nu,\alpha}
F_{\alpha,\lambda,\mu}.
\end{equation*}
The WDVV equation in quantum $K$-theory (see~\cite{WDVV}) shows that the \emph{quantum $K$-product}
defined by
\begin{equation*}
\Phi_{\lambda}
\bullet
\Phi_{\mu}
=\sum_{\nu}
N_{\lambda,\mu}^{\nu}
\Phi_{\nu}
\end{equation*}
is associative. The \emph{quantum $K$-ring} of $X$, denoted by $\qk(X)$, is the $\bb{C}[[q]]$-algebra
$K(X)\otimes\bb{C}[[q]]$ equipped with the quantum $K$-product.
If we take the non-equivariant limit, we obtain the quantum $K$-ring of $X$, denoted by $\qk(X)$.

Givental--Lee showed that the quantum $K$-ring $\qk(X)$ equipped with the
quantized pairing $F_{\alpha,\beta}$
is a (formal) commutative \emph{Frobenius algebra} with unit 1;
see~\cite[Corollary 1]{WDVV}
and~\cite[Proposition 12 (1)]{Lee}.

	\subsection{Stacky structure sheaf}
Recall the standard GIT presentation of
$
X=[M_{r\times N}^{ss}/\text{GL}_r(\bb{C})],
$
where $M_{r\times N}$ is the affine space of $r$ by $N$ complex matrices and $M_{r\times N}^{ss}\subset M_{r\times N}$ is the open subset of matrices of full rank. Let
\begin{equation*}
	\fr{X}=[M_{r\times N}/\text{GL}_r(\bb{C})]
\end{equation*}
be the quotient stack. Let $T\cong(\C^{*})^{N}\subset \mathrm{GL}_{N}(\C)$ be the maximal torus of diagonal matrices and let $\Gamma=\kg(\mathrm{pt})$.
Then there is a homomorphism 
\begin{equation}
	\label{eq:kirwan-map}
	\begin{aligned}
		R(\GLr)
		\rightarrow
		\kg(\fr{X})
		,
		\quad
		\mathrm{V}
		\mapsto
		M_{r\times N}\times_{\GLr}\mathrm{V}.
	\end{aligned}
\end{equation}
Here the $T$-linearization on $M_{r\times N}\times_{\GLr}\mathrm{V}$ is defined via $\lambda\cdot[m,v]=[ \lambda\cdot m,v]$, where $\lambda\in T$, $m\in M_{r\times N}$, $v\in \mathrm{V}$, and $m$ is defined upto right multiplication of $\GLr$.
Since $\kg(\fX)$ is a $\Gamma$-module,
we can extend~\eqref{eq:kirwan-map} to a surjective homomorphism
$R(\GLr)
\otimes
\Gamma
\rightarrow
\kg(\fr{X})$: indeed every elements of $\kg(\fr{X})$ are precisely $\GLr\times T$-equivariant bundle on $M_{r\times N}$, and thus arise from a $\GLr\times T$-module.
We denote by $V \in \kg(X)$ the restriction of $\mathrm{V}$ to $X$, i.e.,
$$V=M^{ss}_{r\times N}\times_{\GLr}\mathrm{V}.$$
We abuse notation by referring to $\mathrm{V}$ as the $K$-theory class in $\kg(\fr{X})$ associated with the $\GLr$-module $\mathrm{V}$,
using Roman font to distinguish classes over the quotient stack $\fX$ from those over the Grassmannian $X$.

Let $\C^{r}$ denote the standard representation of $\GLr$, and let $\mathrm{S}=(\C^{r})^{\vee}$ be
its dual.
Let
\begin{equation}
	\label{eq:stacky-S}
	\mathrm{S}:=M_{r\times N}\times_{\GLr}(\C^{r})^{\vee} 
\end{equation}
be the associated vector bundle to the representation $\mathrm{S}$.
Note that the restriction of $\mathrm{S}$ to $X$ is the tautological subbundle $S$: each fiber of $\mathrm{S}$ is identified to $\bb{C}^r$ mapping to $\bb{C}^n$.

The \emph{$K$-theoretic Kirwan map} is the composition of the homomorphism~\eqref{eq:kirwan-map} with the restriction map $\kg(\fr{X})\rightarrow \kg(X)$. It induces the following linear isomorphism
\begin{equation*}
	\ca{R}^{r,k}\xrightarrow{\sim} K(X),
\end{equation*}
sending the representation $\bS^{\lambda}(\mathrm{S})$
to the vector bundle $\bS^{\lambda}(S)$.
The collection
$
\{
\bS^{\lambda}(S)
\}_{\lambda\in \P_{r,k}}
$
is a $\Z$-basis of the $K$-group $K(X)$ and
a basis of the equivariant $K$-group $\kg(X)$ over $\Gamma=\kg(\mathrm{pt})$.

We now define the ``stacky'' structure sheaves $\mathrm{O}_\lambda$ in the $K$-theory of the stack $\mathfrak{X} $, such that the Kirwan map (see~\cite[\textsection 8]{Buch}) sends $\mathrm{O}_\lambda$ to the Schubert structure sheaf $\mathcal{O}_\lambda \in K(X)$ for all $\lambda \in \mathcal{P}_{r,k}$.

\begin{definition}
	For any partition $\lambda$ (with $r$ nonnegative parts), we define
	the \emph{stacky} Schubert structure sheaf by
	$$
	\mathrm{O}_\lambda:=\bG^{\lambda}(\mathrm{S}^{\vee})\in K(\fr{X}),
	$$
	where $\mathrm{S}$ corresponds to the $\GLr$ representation $(\bb{C}^r)^{\vee}$ (see \eqref{eq:stacky-S}).
\end{definition}

Note that $\ca{O}_\lambda$ equals zero in $K(X)$ if $\lambda$ is not contained in $(k)^r$, while $\mathrm{O}_\lambda$ is a non-zero element in $K(\fr{X})$ for any partition $\lambda$ with at most $r$ parts.

\begin{rem}\label{rem:Quantum_dual}
	Note that the  $K$-theoretic dual $\mathrm{O}_\lambda^*= \mathrm{O}_{\lambda^*}\cdot \det(\mathrm{S})$
	restricts to the classical dual $\O_\lambda^*$ for $\lambda\in \P_{r,k}$ (see \cite[Theorem~3.4.1]{Brion}). However, if we expand $\mathrm{O}_{\lambda}^{*}$ in terms of the stacky Schubert structure sheaves,
	the expansion may involve $\mathrm{O}_\nu$ with $\nu\notin\P_{r,k}$ which restricts to zero in $K(X)$ using \cite[Theorem~8.1]{Buch}. For example, when $r=2$ and $N=4$, we have
		\[
		\mathrm{O}_{(1)}^*=\mathrm{O}_{(2,1)}\cdot \det{\mathrm{S}}=
		\mathrm{O}_{(2,1)}
		-
		\mathrm{O}_{(2,2)}
		-
		\mathrm{O}_{(3,1)}
		+
		\mathrm{O}_{(3,2)},
		\]
		while
		\[
		\O_{(1)}^*
		=\mathrm{O}_{(1)}^*|_X
		=\O_{(2,1)}-\O_{(2,2)}.
		\]
		This shows that in general, the quantized $K$-theoretic duals
		may not live in $K(X)$ and they contain more information than
		their classical counterparts.
\end{rem}

\subsection{Stacky evaluation map}
		Recall that the Quot scheme $\quot_{d}(C,N,r)$ parameterize short exact sequence of sheaves $$0\to S\to\cO_{C}^{\oplus n}\to Q\to 0$$
		where $S$ is a locally free sheaf of rank $r$ and degree $d$. The subscheme of $\quot_{d}(C,N,r)$ parameterizing short exact sequence of vector bundles, i.e. $Q$ is locally free, is isomorphic to the morphism space $\mathrm{Mor}_d(C,\grass)$ of degree $d$ maps from $C$ to $\grass$.

We choose a point $p \in C$. The evaluation of an element $f \in \mathrm{Mor}_d(C, \grass)$ at $p$ is equivalent to restricting the corresponding exact sequence of vector bundles to $p \in C$, which defines the morphism
	\[
	\ev_p : \mathrm{Mor}_d(C, \grass) \to \grass.
	\]
	However, when $Q$ has torsion, the restriction of the short exact sequence does not yield a well-defined map to $\grass$. Instead, define the `stacky' evaluation morphism
	\[
	\wev_p : \quot_d(C, N, r) \to \fr{X}
	\]
	by sending an element $[S \xhookrightarrow{} \cO_C^{\oplus N}] \in \quot_d(C, N, r)$ to the restriction at $p$, which gives a $\GLr$-equivariant map
	\[
	S_p \to \cO_p^{\oplus N} \in M_{r \times N}.
	\]
We note that the vector bundle $\cS_p$ on $\quot_d(C,N,r)$ considered in the introduction is the stacky pullback $\wev_{p}^*(\mathrm{S})$, for the dual of the standard representation $\mathrm{S}\in K(\fr{X})$.

\subsection{Perfect obstruction theories and virtual structure sheaves}
\label{sec:POT}
Let $\QuotC$ be the Quot scheme introduced in Section~\ref{subsec:quot-scheme}.
Let $\ca{S}$ and $\ca{Q}$ be the universal
subbundle and quotient sheaf over
the universal curve $\pi:\QuotC\times C\rightarrow \QuotC$. Then
\[\left(
	R\pi_*\shom(\ca{S},\ca{Q})
	\right)^\vee
	\] defines a perfect obstruction theory (POT) for $\QuotC$~\cite[Theorem 1]{Marian-Oprea}.

In general, consider a projective morphism $\ca{C}\rightarrow B$ over a scheme $B$,
where the fibers are prestable curves. Let $\mathrm{Q}_{\ca{C}/B}$ denote the relative
Quot scheme, parametrizing $B$-flat families of coherent quotients of rank-$N$ trivial
sheaves on the fibers. Define $\mathrm{Q}'\subset \mathrm{Q}_{\ca{C}/B}$ as the open locus
corresponding to locally free subsheaves.
Consider the universal sequence
\[
0\rightarrow \ca{S} \rightarrow \O^{\oplus N}_{\ca{C}_{\mathrm{Q}'}}
\rightarrow \ca{Q}\rightarrow 0
\]
over the universal curve $\pi:\ca{C}_{\mathrm{Q}'}:=\mathrm{Q}'\times_B\ca{C}\rightarrow B$.
According to the proof of~\cite[Theorem 2]{MOP},
\begin{equation}
	\label{eq:POT}
	\left(
		R\shom_\pi (\ca{S},\ca{Q})
	\right)^\vee
	=	\left(
		R\pi_*\shom(\ca{S},\ca{Q})
		\right)^\vee
\end{equation}
defines a (2-term) perfect obstruction theory relative to $\nu:\mathrm{Q}'\rightarrow B$.
If we consider the $T\cong (\C^*)^N$-action on $\mathrm{Q}'$,
induced by the action on $\O^{\oplus N}$,
and the trivial $T$-action on $B$,
the complex~\eqref{eq:POT} defines a $T$-equivariant POT for the morphism
$\nu_T:[\mathrm{Q}'/T]\rightarrow [B/T]$.

Note that the injection $\ca{S} \rightarrow \O^{\oplus N}_{\ca{C}_{\mathrm{Q}'}}$
induces a map
$$
\ca{C}\xrightarrow{f} \fr{X},
$$ where $\fr{X}=[M_{r\times N}/\GLr]$ is the quotient stack containing the Grassmannian $X$.
Since the tangent complex $\bb{T}_{\fX}$ of $\fr{X}$ is quasi-isomorphic to $[\glr\otimes \O_{M_{r\times N}}
\rightarrow T_{M_{r\times N}}]$ of amplitude in [-1,0] (see \cite[4.29]{Chen_orbifold}),
we have
\[
Lf^*\bb{T}_{\fX}\simeq [ \shom(\ca{S},\ca{S}) \rightarrow (\ca{S}^\vee)^{\oplus N}]	\simeq
\shom(\ca{S},\ca{Q}).
\]
Hence, the POT~\eqref{eq:POT} can also be written as
\begin{equation}
	\label{eq:alternate-POT}
	\left(
		R\pi_*Lf^*\bb{T}_{\fX}
	\right)^\vee,
\end{equation}
which resembles that for stable maps (see, e.g.~\cite{Behrend}).

By~\cite[Definition 2.2]{Qu}, the relative POT~\eqref{eq:POT} (or, equivalently,~\eqref{eq:alternate-POT})
induces virtual pullbacks
\[
	\nu^!: K_0(B) \rightarrow K_0(\mathrm{Q}')
	\quad
	\text{and}
	\quad
	\nu_T^!: \kG(B) \rightarrow \kG(\mathrm{Q}'),
\]
such that the ($T$-equivariant) virtual structure sheaf on the Quot scheme equals
$$
	\ovir_{\QuotC}:= \nu^!(\O_{\spec\C})	\in \kG(\QuotC)
$$
as the virtual pullback of $\O_{\spec \C}$ along $\nu$.

\section{quantum-$K$ TQFT}\label{sec:TQFT}
In this section, we prove Theorem~\ref{thm:Intro_structure_constants} and Theorem~\ref{thm:Intro_TQFT},
assuming the wall-crossing result from Corollary~\ref{thm:3-point_invariants_QK_equals_Quot} and the degeneration formulas
from Theorem~\ref{thm:intro_degeneration_formula}.
We make critical use of a vanishing theorem for the Quot scheme from an adjoining paper~\cite{SinhaZhang2}. Notably, this vanishing result, proven using Atiyah--Bott localization on the Quot scheme over $\bP^1$, is logically independent of the findings in this paper.

\subsection{Orthogonality}

Recall that $\ca{R}^{r,\ell}$ in Definition~\ref{def:subspace-rep-ring} is the subspace of the representation ring $R(\GLr)$ spanned by
$\{
\bS^{\lambda}(\mathrm{S}) 
\mid \lambda\in \P_{r,\ell}
\}$. The classical Littlewood-Richardson rule implies the following very useful property: For any $\mathrm{V}\in \ca{R}^{r,\ell}$ and $\mathrm{W}\in\ca{R}^{r,m}$, the product
\begin{equation}\label{eq:LR_rule}
\mathrm{V}\cdot \mathrm{W}\in \ca{R}^{r,\ell+m}
\end{equation}
belongs to the span of terms indexed by partitions in the $r\times (\ell+m)$ rectangle. We use the the vanishing result in Theorem~\ref{thm:vanishing} to prove the following results.
\begin{cor}
	\label
	{cor:three-point_quot}
	For any $\mathrm{V},\mathrm{W}\in \ca{R}^{r,k} \cong K(X)$ and $d\geq m>0$, we have
	\begin{align*}
		\langle
		\det(\mathrm{S})^{m}
		,
		\mathrm{V}
		,		
		\mathrm{W}
		\rangle^{\quot}_{0,d} 
		=0.
	\end{align*} 
\end{cor}
\begin{proof}
	 Note that $\det(\mathrm{S})^{m}=\bS^{(m)^{r}}(\mathrm{S})$, where $(m)^r=(m,\dots,m)$.
Therefore,
	\[
	\det(\mathrm{S})^{m}
	\cdot
	\mathrm{V}
	\cdot		
	\mathrm{W}
\in \ca{R}^{r,m+2k}.
	\]
Since $\mathrm{S}$ has rank $r$, the Littlewood-Richardson rule/Pieri rule implies that $\det(\mathrm{S})\cdot \bS^{\lambda}(\mathrm{S})=\bS^{\nu}(\mathrm{S})$ where $\nu=(\nu_1,\dots,\nu_r)= (\lambda_1+1,\dots\lambda_r+1)$. Each term $\bS^{\nu}(\mathrm{S})$ in the Schur expansion of $\det(\mathrm{S})^{m}
\cdot
\mathrm{V}
\cdot		
\mathrm{W}$ satisfies $\nu\in \ca{R}^{r,m+2k} $ and $\nu_r\ge m>0$. For all such $\nu$, part(ii) of Theorem~\ref{thm:vanishing} implies 
$
\langle
\bb{S}^{\nu}(\mathrm{S})
\rangle^{\quot}_{0,d} = 0
$
whenever $d\ge m$.
\end{proof}
By setting $m=1$ in the above corollary and using Theorem~\ref{thm:intro-Wall_crosing},
we obtain an alternate proof of a conjecture of
Buch--Mihalcea~\cite[Conjecture 4.3]{GMSXZZ} in the case of
Grassmannians.
This result may be thought of as a \emph{divisor equation} for the $K$-theoretic Gromov-Witten invariants: For $V,W\in \kg(X)$ and $d>0$, the quantum $K$-invariants satisfy
\[
\langle
\ca{O}_{(1)},
V,W
\rangle_{0,3,d}
^{}
=
\langle
V,W\rangle_{0,2,d}
^{}
.
\]

For more general homogeneous spaces, the $K$-theoretic `divisor equation' conjectuerd by Buch--Mihalcea
plays crucial roles in studying quantum $K$-product.
For example, it is used in the proof of the Chevalley formula for the equivariant quantum $K$-theory of cominuscule varieties~\cite[Theorem 3.9]{BCPMP2}.
Gu--Mihalcea--Sharpe--Xu--Zhang--Zou~\cite{GMSXZZ}
also used it to (partially) verify the conjectured complete set of relations for the equivariant quantum $K$-ring of the partial flag manifolds.

\begin{prop}[Orthogonality]
	\label
	{prop:quantized-dual-basis}
	We have
	\[
	N(0)^{\nu}_{\lambda}=\langle\!\langle
	\mathrm{O}_{\lambda}
	,
	\mathrm{O}_{\nu}^{*}
	\rangle\!\rangle^{\quot}_0
	=\delta_{\lambda,\nu}
	\]
	for $\nu,\lambda\in \P_{r,k}$.
\end{prop}
\begin{proof}
	The restriction of $\mathrm{O}_{\nu}^{*}$ to $\grass$ is the classical
	$K$-theoretic dual $\O_{\nu}^{*}=(1-\O_{(1)})\cdot \O_{\nu^{*}}$ defined in~\cite[\textsection{5.1}]{Buch-Mihalcea}
	satisfying
	\begin{equation}
	\label
	{eq:classic-dual}
	\chi(X,\O_{\nu}\cdot \O_{\lambda}^{*})
	=\delta_{\nu,\lambda}
	.
	\end{equation}The degree zero part of the identity follows from~\eqref{eq:classic-dual} and the identity $(1-\O_{(1)})=\det(S)$.
	For $d>0$, we have
	\[
	\langle
	\mathrm{O}_{\lambda}
	,
	\mathrm{O}_{\nu}^{*}
	\rangle^{\quot}_{0,d}
	=
	\langle
	\mathrm{O}_{\lambda}
	,
	\mathrm{O}_{\nu^*},
	\det(\mathrm{S})
	\rangle^{\quot}_{0,d}
	=0,
	\]
	where the last equality follows from Corollary~\ref{cor:three-point_quot}.
\end{proof}
\begin{rem}
	Note that $\mathrm{O}_\lambda^*$ may
	not correspond to a representation in $\ca{R}^{r,k}$.
	In fact, $\operatorname{Span}\{
	\mathrm{O}_\lambda^* 
	\mid \lambda\in \P_{r,k}\}\subset K(\fr{X})$
	correspond to $
	\operatorname{Span}\{
	\bS^{\lambda}(\mathrm{S}) 
	\mid \lambda\in \P_{r,k+1} \text{ and }
	\lambda_r>0\}
	\subset R(\GLr).
	$
	Nevertheless, the above proposition shows that $\{\mathrm{O}_{\lambda}^{*}
	\}$ are dual to $\{\mathrm{O}_{\lambda}\}$ with respect to the pairing
	$\langle\!\langle
		-
		,
-
		\rangle\!\rangle^{\quot}_0$. 
\end{rem}
 
\subsection{Structure constants}
Corollary~\ref{thm:3-point_invariants_QK_equals_Quot} states that
$F_{\lambda,\mu,\nu} = N(0)_{\lambda,\mu,\nu}$ and $ F_{\lambda,\mu} = N(0)_{\lambda,\mu}$ for any partitions $\lambda,\mu,\nu\in \P_{r,k}$. Below we give a description of the inverse of the  quantized pairing matrix $(F^{\alpha,\beta})$.

\begin{cor}
	\label
	{cor:inverse-quantized-pairing}
	We have
	\begin{enumerate}[\normalfont(i)]
		\item $
		F^{\lambda,\nu}
		=N(0)^{\lambda,\nu}$.
		\item
		$F^{\lambda,\nu}$ is a linear function in $q$.
	\end{enumerate}
\end{cor}

\begin{proof}
	 It follows from the degeneration formula in Theorem~\ref{thm:intro_degeneration_formula} that
	\begin{align*}
	N(0)^{\lambda,\nu}=\langle\!\langle
	\mathrm{O}_{\lambda}^{*}
	,
	\mathrm{O}_{\nu}^{*}
	\rangle\!\rangle^{\quot}_{0}
	=
	\sum_{\alpha,\beta\in \P_{r,k}}
	\langle\!\langle
	\mathrm{O}_{\lambda}^{*}
	,
	\mathrm{O}_{\alpha}
	\rangle\!\rangle^{\quot}_{0}
	F^{\alpha,\beta}
	\langle\!\langle
	\mathrm{O}_{\beta},
	\mathrm{O}_{\nu}^{*}
	\rangle\!\rangle^{\quot}_{0}
	=
	F^{\lambda,\nu}.
	\end{align*}
	Here we have used orthogonality property  of Proposition~\ref{prop:quantized-dual-basis} in the last equality.
	(ii) follows from the fact 
	\[
	\langle
	\mathrm{O}_{\nu}^{*}
	,
	\mathrm{O}_{\lambda}^{*}
	\rangle^{\quot}_{0,d}
	=
	\langle
	\mathrm{O}_{\nu^{*}}
	,
	\mathrm{O}_{\lambda^{*}}
	,
	\det(\mathrm{S})^{2}
	\rangle^{\quot}_{0,d}
	=0,
	\]
	for any $d\geq 2$ by Corollary~\ref{cor:three-point_quot}.
\end{proof}
\begin{rem}
	For a finite-dimensional Frobenius algebra, the formula in Corollary~\ref{cor:inverse-quantized-pairing} (i)
	follows trivially from the orthogonality in Proposition~\ref{prop:quantized-dual-basis}. Here, in the infinite-dimensional setting, the proof relies on the degeneration formula.
	The above corollary also implies that $N(0)^{\lambda,\mu}$ is the inverse of the pairing matrix $N(0)_{\lambda,\mu}$. 
\end{rem}

As further applications of the quantized $K$-theoretic duals, we  finish the proof of Theorem~\ref{thm:Intro_structure_constants}, i.e. show that the structure constant of the quantum $K$-product $N_{\lambda,\mu}^{\nu}$ equals $N(0)_{\lambda,\mu}^{\nu}$.
\begin{proof}[Proof of Theorem~\ref{thm:Intro_structure_constants}]
Recall that the structure constant is given by the sum $$	N_{\lambda,\mu}^{\nu}
=
\sum_{\alpha\in \P_{r,k}}
F_{\lambda,\mu,\alpha}
F^{\alpha,\nu}.$$
Using Corollary~\ref{thm:3-point_invariants_QK_equals_Quot}, we obtain
	\begin{align*}
		N_{\lambda,\mu}^{\nu}
		&=
		\sum_{\alpha\in \P_{r,k}}
		\langle\!\langle
		\mathrm{O}_\lambda
		,
		\mathrm{O}_\mu
		,
		\mathrm{O}_\alpha
		\rangle\!\rangle^{\quot}_0
		F^{\alpha,\nu}
		\\
		&=
		\sum_{\alpha,\beta\in \P_{r,k}}
		\langle\!\langle
		\mathrm{O}_\lambda
		,
		\mathrm{O}_\mu
		,
		\mathrm{O}_\alpha
		\rangle\!\rangle^{\quot}_0
		F^{\alpha,\beta}
		\langle\!\langle
		\mathrm{O}_\beta
		,
		\mathrm{O}^*_{\nu}
		\rangle\!\rangle^{\quot}_0
		\\
		&= \langle\!\langle
		\mathrm{O}_\lambda
		,
		\mathrm{O}_\mu
		,
		\mathrm{O}^*_{\nu}
		\rangle\!\rangle^{\quot}_0		
		= N(0)_{\lambda,\mu}^{\nu}.
	\end{align*}
	Here we used the degeneration formula from Theorem~\ref{thm:intro_degeneration_formula} in the second-to-last equality.

\end{proof}

The formula of the structure constant in the above corollary
resembles that in a Frobenius algebra. It is not only more transparent and compact but also allows
us to give intuitive proofs of some results of Buch--Mihalcea~\cite{Buch-Mihalcea}.

\begin{cor}[{\cite[Theorem 5.13]{Buch-Mihalcea}}]
	The structure constants satisfy the $S_3$-symmetry, i.e., $N_{\lambda,\mu}^\nu=N_{\lambda,\nu^*}^{\mu^*}$.
	
\end{cor}
\begin{proof}
	It follows from Theorem~\ref{thm:Intro_structure_constants} and Definition~\ref{defn:intro-TQFT-invariants} that
	\[
	\langle\!\langle
	\mathrm{O}_\lambda
	,
	\mathrm{O}_\mu
	,
	\mathrm{O}^*_{\nu}
	\rangle\!\rangle^{\quot}_{0}
	=
	\langle\!\langle
	\mathrm{O}_\lambda
	\cdot
	\mathrm{O}_\mu
	\cdot
	\mathrm{O}_{\nu^*}
	\cdot \det(\mathrm{S})
	\rangle\!\rangle^{\quot}_{0}.
	\]
	This easily implies the $S_3$-symmetry.
\end{proof}
Part (i) of Theorem~\ref{thm:vanishing} implies that the structure constants
$N_{\lambda,\mu}^{\nu}$ are polynomials in the Novikov variable $q$ of degree at most $2(N-r)$.
In fact, by using the stronger vanishing result in part (ii) of Theorem~\ref{thm:vanishing},
we can obtain the following optimal bounds of the degree of $N_{\lambda,\mu}^{\nu}$:
\begin{cor}[{\cite[Corollary 5.8]{Buch-Mihalcea}}]
	\label
	{cor:finiteness-quantum-product}
	The structure constants
	$N_{\lambda,\mu}^{\nu}$ are polynomials in $q$ of degree $\leq \min\{r,k\}$ with $k=N-r$.
\end{cor}
\begin{proof}
	Since $\qk(\grass)\cong\qk(\mathrm{Gr}(k,N))$, we only need to show that the degree of $N_{\lambda,\mu}^{\nu}$
	is less than or equal to $k$.
	Using a similar argument to the proof of Corollary~\ref{cor:three-point_quot}, each term $\bS^{\nu}(\mathrm{S})$
	in the Schur exapansion of 
	$
	\mathrm{O}_\lambda
	\cdot
	\mathrm{O}_\mu
	\cdot
	\mathrm{O}_{\nu^*}
	\cdot \det(\mathrm{S})$ is contained in $\ca{R}^{r,3k+1}$, and $\nu$ satisfies the condition $\nu_r>0$. For all such $\nu$, part(ii) of Theorem~\ref{thm:vanishing} implies 
$
\langle
\bb{S}^{\nu}(\mathrm{S})
\rangle^{\quot}_{0,d} = 0
$
whenever $d\ge k+1$.
\end{proof}
\begin{rem}
	The property proved in the previous corollary is often referred to as~\emph{finiteness} of the quantum $K$-product.
	It was first proved for Grassmannians in~\cite{Buch-Mihalcea} and then generalized to cominuscule varieties~\cite{BCPMP,BCPMP3} and flag manifolds~\cite{Kato,ACT,ACTH}.
	As shown in the proof of the above corollary, the finiteness is a consequence of
	the vanishing of certain $K$-theoretic Quot scheme invariants.
\end{rem}
\subsection{Proof of Theorem~\ref{thm:Intro_TQFT}}
	\begin{lem}
For any $\mathrm{V},\mathrm{W}\in K(\fr{X})$ and $g_1+g_2=g$,
\begin{equation}\label{eq:gluing_formula_inside_proof}
\langle\!\langle
\mathrm{V}, \mathrm{W} \rangle\!\rangle^{\quot}_{g} = \sum_{\rho\in \P_{r,k}}^{}\langle\!\langle
\mathrm{V}, \mathrm{O}_{\rho} \rangle\!\rangle^{\quot}_{g_1}\cdot \langle\!\langle
\mathrm{O}^*_{\rho}, \mathrm{W} \rangle\!\rangle^{\quot}_{g_2}.
\end{equation}
\end{lem}
\begin{proof}
	We repeatedly use the degeneration formula in Theorem~\ref{thm:intro_degeneration_formula}.
\begin{align*}
\langle\!\langle
\mathrm{V}, \mathrm{W} \rangle\!\rangle^{\quot}_{g}&=\sum_{\rho,\beta\in \P_{r,k}}^{}\langle\!\langle
\mathrm{V}, \mathrm{O}_{\rho} \rangle\!\rangle^{\quot}_{g_1}\cdot F^{\rho,\beta}\cdot\langle\!\langle
\mathrm{O}_{\beta}, \mathrm{W} \rangle\!\rangle^{\quot}_{g_2}\\
&=\sum_{\rho\in \P_{r,k}}^{}\langle\!\langle
\mathrm{V}, \mathrm{O}_{\rho} \rangle\!\rangle^{\quot}_{g_1}\cdot\sum_{\alpha,\beta\in \P_{r,k}} \langle\!\langle
\mathrm{O}^*_{\rho}, \mathrm{O}_{\alpha} \rangle\!\rangle^{\quot}_{0} \cdot F^{\alpha,\beta}\cdot\langle\!\langle
\mathrm{O}_{\beta}, \mathrm{W} \rangle\!\rangle^{\quot}_{g_2}
\\
&=	\sum_{\rho\in \P_{r,k}}^{}\langle\!\langle
\mathrm{V}, \mathrm{O}_{\rho} \rangle\!\rangle^{\quot}_{g_1}\cdot \langle\!\langle
\mathrm{O}^*_{\rho}, \mathrm{W} \rangle\!\rangle^{\quot}_{g_2}.
\end{align*}
Here we used the orthogonality relation in Proposition~\ref{prop:quantized-dual-basis} for the second equality.
\end{proof}
		To prove the gluing relation in Theorem~\ref{thm:Intro_TQFT}, it is enough to show the following two gluing formulas:
	\begin{itemize}
		\item Let $\underline{\lambda}=( \lambda^1,\dots ,\lambda^{s})$ and $\underline{\nu}=(\nu^1,\dots ,\nu^{t})$ be tuples of partitions in $\P_{r,k}$. Then for any $g_1+g_2=g$ and for any splitting of tuples $\underline{\lambda} = \underline{\lambda_1}\cup\underline{\lambda_2}$ and $\underline{\nu} = \underline{\nu_1}\cup\underline{\nu_2}$,
			\begin{align*}
		N(g)_{\underline{\lambda}}^{\underline{\nu}}=\sum_{\rho\in\P_{r,k}}N(g_1)_{\underline{\lambda_1},\rho}^{\underline{\nu_1}}\cdot N(g_2)_{\underline{\lambda_2}}^{\underline{\nu_2},\rho}.
		\end{align*}
	\begin{proof}
Set $\mathrm{V}= \mathrm{O}_{\underline{\lambda_1}}\cdot \mathrm{O}^*_{\underline{\nu_1}}$ and $\mathrm{W}= \mathrm{O}_{\underline{\lambda_2}}\cdot \mathrm{O}^*_{\underline{\nu_2}}$ in \eqref{eq:gluing_formula_inside_proof}.
		\end{proof}

		\item For any tuple of partitions $\underline{\lambda}$ and $\underline{\nu}$ and $g>0$	
		\begin{align*}
		N(g)_{\underline{\lambda}}^{\underline{\nu}}=\sum_{\rho\in\P_{r,k}}N(g-1)_{\underline{\lambda},\rho}^{\underline{\nu},\rho}.
		\end{align*} 
		\begin{proof}
			The second degeneration formula in Theorem~\ref{thm:intro_degeneration_formula} implies that 
			\begin{align*}
				\langle\!\langle
			\mathrm{V} \rangle\!\rangle^{\quot}_{g}&=\sum_{\alpha,\beta\in\P_{r,k}} \langle\!\langle
			\mathrm{V},
			\mathrm{O}_{\alpha}, \mathrm{O}_{\beta}
			\rangle\!\rangle^{\quot}_{g-1}\cdot F^{\alpha,\beta}\\
			&=\sum_{\alpha,\beta \in\P_{r,k}} \langle\!\langle
			\mathrm{V},
			\mathrm{O}_{\alpha}, \mathrm{O}_{\beta}
			\rangle\!\rangle^{\quot}_{g-1}\cdot \langle\!\langle
			\mathrm{O}^*_{\beta}, \mathrm{O}^*_{\alpha}
			\rangle\!\rangle^{\quot}_{0}
			\\
			&= \sum_{\alpha \in\P_{r,k}} \langle\!\langle
			\mathrm{V},
			\mathrm{O}_{\alpha}, \mathrm{O}^*_{\alpha}
			\rangle\!\rangle^{\quot}_{g-1}.
			\end{align*}
			Here we used Corollary~\ref{cor:inverse-quantized-pairing} in the second equality, and \eqref{eq:gluing_formula_inside_proof} for the last equality.
			Set $V=\mathrm{O}_{\underline{\lambda}}\cdot \mathrm{O}^*_{\underline{\nu}}$ to obtain the required gluing formula.
		\end{proof}
	\end{itemize}

\section{Quantum reduction map}
 \label
 {sec:quantum-reduction}
 In this section, we introduce the quatum reduction map $\kappa$ and prove
 it is a ring homomorphism, as described in Theorem~\ref{thm:intro_kappa_homomorphism}.
We then explore some applications, including re-deriving the Whitney presentation
 of the quantum $K$-ring $\qk(X)$.
\subsection{Ring homomorphism property}
Recall that the $K$-theoretic Kirwan map is the following isomorphism 
\[
R(\GLr)
\rightarrow
K(\fr{X})
,
\quad
\mathrm{V}
\mapsto
M_{r\times N}\times_{\GLr}
\mathrm{V},
\]
sending a representation to its associated $K$-theory class on $\fr{X}$.
We define the following deformation of the $K$-theoretic Kirwan map: 
\begin{definition}
	The quantum reduction map $\kappa: 
	R(\GLr)
	\otimes
	\bb{Z}[[q]]
	\ra
	\qk(X)$ is defined to be the $\bb{Z}[[q]]$-linear extension of the formula
	\begin{equation}
		\label
		{eq:kappa-map}
		\begin{aligned}
			\kappa(\mathrm{V})=
			\sum_{\alpha\in\P_{r,k}}
			\langle\!\langle
			\mathrm{V}
			,
			\mathrm{O}^*_{\alpha}
			\rangle\!\rangle^{\quot}_0
			\ca{O}_{\alpha}
		\end{aligned},
	\end{equation}
	where $\mathrm{V}$, by abuse of notation, also denotes the $K$-theory class on $\fr{X}$ associated to $\mathrm{V}\in R(\GLr)$.
\end{definition}

\begin{prop}
	\label
	{prop:kappa-identity}
	
	For any $\mathrm{V}\in \ca{R}^{r,k}$, we have $\kappa(\mathrm{V})=V $ where
	$$V=M^{ss}_{r\times N}\times_{\GLr}
	\mathrm{V}\in K(X).$$
\end{prop}

\begin{proof}
	Recall that $\ca{R}^{r,k}$ is spanned by $\{\mathrm{O}_{\lambda}:\lambda\in \P_{r,k}\}$. The statement follows from Proposition~\ref{prop:quantized-dual-basis} and the fact that $\mathrm{O}_{\lambda}$ restricts to the class of structure sheaf $\ca{O}_\lambda$ over $X$.
\end{proof}
For any $\mathrm{V}\in R(\GLr)$, the element $\kappa(\mathrm{V})\in \qk(X)$ can be regarded as an element of $\bb{Z}[[q]]\otimes K(\fr{X})$ by replacing $\ca{O}_{\nu}$ with $\mathrm{O}_{\nu}$ in~\eqref{eq:kappa-map}. This identification is an abuse of notation: we use the linear map
\[
\qk(X)\cong \ca{R}^{r,k}\otimes \bb{Z}[[q]] \xhookrightarrow{} R(\GLr)\otimes \bb{Z}[[q]],
\]
which is a right inverse to $\kappa$.
\begin{lemma}
	\label
	{lem:kappa-trivial}
	Let $\mathrm{V}\in R(\GLr)\cong K(\fr{X})$ be the associated $K$-theory class.
	We have
	\[
	\langle\!\langle
	\kappa(\mathrm{V})
	,
	\mathrm{W}
	\rangle\!\rangle^{\quot}_0
	=
	\langle\!\langle
	\mathrm{V}
	,
	\mathrm{W}
	\rangle\!\rangle^{\quot}_0
	\]
	for any $\mathrm{W}\in K(\fr{X})$.

\end{lemma}

\begin{proof}
	It follows from the definition of $\kappa$ (see  \eqref{eq:kappa-map}) and the identity \eqref{eq:gluing_formula_inside_proof} that
	\begin{align*}
		\langle\!\langle
		\kappa(
		\mathrm{V})
		,
		\mathrm{W}
		\rangle\!\rangle^{\quot}_0
		=\sum_{\alpha\in \P_{r,k}}
		\langle\!\langle
		\mathrm{V},\mathrm{O}^*_{\alpha}
		\rangle\!\rangle^{\quot}_0
		\langle\!\langle
		\mathrm{O}_{\alpha}
		,
		\mathrm{W}
		\rangle\!\rangle^{\quot}_0
		=
		\langle\!\langle
		\mathrm{V}
		,
		\mathrm{W}
		\rangle\!\rangle^{\quot}_0.
	\end{align*}
\end{proof}

\begin{theorem}
	\label
	{thm:kappa-homo}
	The quantum reduction map $\kappa$ is a surjective ring homomorphism.
\end{theorem}
\begin{proof}
	It is clear from the definition that $\kappa$ is surjective and additive.
	Let $\mathrm{V},\mathrm{W}\in R(\GLr)\cong K(\fr{X})$,
	then by the definition of the quantum product, Theorem~\ref{thm:Intro_structure_constants}, and repeatedly using the identity in \eqref{eq:gluing_formula_inside_proof}, we obtain
	\begin{align*}
\kappa(\mathrm{V})\bullet \kappa(\mathrm{W})&= \sum_{\alpha,\beta\in\P_{r,k}}\langle\!\langle
\mathrm{V},\mathrm{O}^*_{\alpha}
\rangle\!\rangle^{\quot}_0\cdot \langle\!\langle
\mathrm{W},\mathrm{O}^*_{\beta}\rangle\!\rangle^{\quot}_0
\mathcal{O}_\alpha\bullet\ca{O}_{\beta}
\\
&=\sum_{\alpha,\nu\in\P_{r,k}}\langle\!\langle
\mathrm{V},\mathrm{O}^*_{\alpha}
\rangle\!\rangle^{\quot}_0\cdot \langle\!\langle
\mathrm{W},\mathrm{O}^*_{\beta}\rangle\!\rangle^{\quot}_0
\langle\!\langle
\mathrm{O}_\alpha,\mathrm{O}_{\beta},\mathrm{O}^*_{\nu}\rangle\!\rangle^{\quot}_0\ca{O}_{\nu}\\
&= \sum_{\alpha,\beta,\nu\in\P_{r,k}}\langle\!\langle
\mathrm{V},\mathrm{O}^*_{\alpha}
\rangle\!\rangle^{\quot}_0\cdot \langle\!\langle
\mathrm{W},
\mathrm{O}_\alpha,\mathrm{O}^*_{\nu}\rangle\!\rangle^{\quot}_0\ca{O}_{\nu}
\\
&= \sum_{\nu\in\P_{r,k}}\langle\!\langle
\mathrm{V},
\mathrm{W},
\mathrm{O}^*_{\nu}\rangle\!\rangle^{\quot}_0\ca{O}_{\nu}\\
&= \kappa(\mathrm{V}\cdot \mathrm{W}).
	\end{align*}
\end{proof}
\begin{rem}
	Let $G$ be a complex reductive group and let $Z$ be a smooth projective or convex quasi-projective
	polarized $G$-variety.
	In~\cite{Gonzalez-Woodward}, Gonz\'alez and Woodward
	used $K$-theoretic \emph{affine gauged Gromov-Witten invariants} to
	define the
	\emph{linearized quantum Kirwan map}
	\[
	D_{\alpha}\kappa_Z^G:\qk_G(Z)
	\rightarrow
	\qk(Z\git G)
	\]
	for any $\alpha$ in the $G$-equivariant quantum $K$-ring $\qk_G(Z)$.
	They showed that this map is a ring homomorphism preserving the (big) quantum $K$-products.
	We expect the non-equivariant limit of our reduction map $\kappa$ to coincide with $D_0\kappa_{M_{r\times N}}^{\GLr}$
	in the case of Grassmannian and leave it for future investigation.
\end{rem}

\begin{rem}
	Consider the extended representation ring $R(\GLr)
\otimes\bb{Z}[[q]]$, equipped with the pairing
\begin{equation}
	\label
	{eq:quot-pairing}
	\mathrm{V}\otimes \mathrm{W} \mapsto
	\langle\!\langle
	\mathrm{V}
	,
	\mathrm{W}
	\rangle\!\rangle^{\quot}_0
	\quad \mathrm{V},\mathrm{W}\in R(\GLr).
\end{equation}
Note that this does not define a Frobenius algebra because the representation ring $R(\GLr)$ is
infinite-dimensional and the pairing~\eqref{eq:quot-pairing} is
degenerate according to Lemma~\ref{lem:kappa-trivial}. It follows from Lemma~\ref{lem:kappa-trivial} and Corollary~\ref{thm:3-point_invariants_QK_equals_Quot}
that the quantum reduction map $\kappa$ preserves the pairings, i.e.,
\[
\langle\!\langle
\mathrm{V}
,
\mathrm{W}
\rangle\!\rangle^{\quot}_0
=
\langle\!\langle
\kappa(\mathrm{V})
,
\kappa(\mathrm{W})
\rangle\!\rangle_{0,2}
\]
for $\mathrm{V},\mathrm{W}\in R(\GLr)$.
\end{rem}
\subsection{Ring presentation for $\qk(X)$}
The product structure in quantum $K$-ring $\qk(X)$ can be derived from the ring structure of the representation ring $R(\GLr)$, i.e., Littlewood-Richardson rule, and the quantum reduction map $\kappa$ defined in Section~\ref{sec:quantum-reduction}:
\begin{equation*}
\ca{O}_\lambda\bullet\ca{O}_\mu = \kappa(\mathrm{O}_\lambda)\bullet \kappa(\mathrm{O}_{\mu}) = \sum_{\nu\in \P_{r,2k}}c_{\lambda,\mu}^{\nu}\kappa(\mathrm{O}_{\nu}),
\end{equation*}
where $c_{\lambda,\mu}^{\nu}$ are the $K$-theoretic Littlewood-Richardson coefficients. Using this strategy, we obtain a close formula for the structure constants of
$\qk(\mathrm{Gr}(2,N))$ in \cite{SinhaZhang2}.

We now use our technique to give an alternate proof of the quantum $K$-ring relations in~\cite{GMSZ2}.
Let  $0\to S\to  \ca{O}_X^{\oplus N}\to  Q\to 0 $ denote the universal sequence on $\grass$.
Note that $\det(Q)=\det(S^\vee)$. Using the classical orthogonality in \eqref{eq:classic-dual} and the fact that $\chi(X,\O_\nu)=1$ for all $\nu\in \P_{r,k}$,
one can show that
\begin{equation}
	\label
	{eq:det-quotient}
	\det(S^\vee)
	=
	\sum_{\lambda\in\P_{r,k}} \O_{\lambda}\in K(X).
\end{equation}

\begin{prop}
\label
{prop:inverse-det-kappa-image}
\[
\kappa(
\det(\mathrm{S}^{\vee})
)
=\frac{\det(Q)}{1-q}.
\]
\end{prop}
\begin{proof}

For each $\lambda\in \P_{r,k}$,
\[
\langle\!\langle
\det(\mathrm{S}^{\vee})
,
\mathrm{O}^*_{\lambda}
\rangle\!\rangle^{\quot}_0
=
\langle\!\langle
\det(\mathrm{S}^{\vee})
\cdot
\mathrm{O}_{\lambda^*}\cdot 
\det(\mathrm{S})
\rangle\!\rangle^{\quot}_0
=
\langle\!\langle
\mathrm{O}_{\lambda^*}
\rangle\!\rangle^{\quot}_0
=\frac{1}{1-q}.
\]
Then the proposition follows from~\eqref{eq:det-quotient} and the definition of $\kappa$.
\end{proof}
Recall that $\ca{R}^{r,\ell}$ denotes the subspace of the representation ring $R(\GLr)$ spanned by
$\{
\bS^{\lambda}(\mathrm{S}) 
\mid \lambda\in \P_{r,\ell}
\}$. We will use the property~\eqref{eq:LR_rule} of the Littlewood-Richardson rule several times below.
\begin{cor}\label{cor:quant_prod_Wedge^i_and_detQ}
	 For $0<i\le r$, the quantum $K$-product
\begin{equation*}
	\wedge^i S\bullet \det Q=
		(1-q)\wedge^i S\cdot \det Q.
\end{equation*}
\end{cor}
\begin{proof}
Using Proposition~\ref{prop:inverse-det-kappa-image},	
$$\wedge^i S\bullet \det Q =\kappa(\wedge^i \mathrm{S})\bullet(1-q) \kappa(\det (\mathrm{S}^{\vee}))= (1-q)	\kappa(\wedge^i \mathrm{S}\cdot \det (\mathrm{S}^{\vee})).$$
To show that $\kappa(\wedge^i \mathrm{S}\cdot \det (\mathrm{S}^{\vee})) = \wedge^i S\cdot \det(S)^{-1}$, it suffices to prove that 
\begin{align*}
\langle
\wedge^i\mathrm{S},\det (\mathrm{S}^{\vee})
,
\mathrm{O}^*_{\nu}
\rangle^{\quot}_{0,d}=
\langle
\wedge^i\mathrm{S}
,
\mathrm{O}_{\nu^*}
\rangle^{\quot}_{0,d}= 0.
\end{align*}
for each $\nu\in \P_{r,k}$, $d\ge1$. Note that $\mathrm{O}_{\nu^*}\in \ca{R}^{r,k}$, thus $\wedge^i\mathrm{S}\cdot 
\mathrm{O}_{\nu^*} \in \ca{R}^{r,k+1}$.
Part (i) of Theorem~\ref{thm:vanishing} implies that $\langle \mathrm{V}\rangle^{\quot}_{0,d}=0$
for any $\mathrm{V}\in\ca{R}^{r,k+1}$ and $d\ge1$, thus completing the proof.
\end{proof}

We can express $\wedge^j Q$ in the Schur basis for $K(X)$ by extracting the coefficient of $y^j$ in
$$
\wedge_y  Q=\frac{\wedge_y(\ca{O}_X^{\oplus N})}{\wedge_y S}=\frac{(1+y)^N}{1+yS+y^2\wedge^2S+\cdots y^r\wedge^r S}.
$$
This coefficient is a polynomial in the exterior powers of $S$ of degree at most $j$.
Proposition~\ref{prop:kappa-identity} and \eqref{eq:LR_rule} imply that there exists a unique class $\mathrm{H}_j\in \ca{R}^{r,j}$ such that $\kappa(\mathrm{H}_j) = \wedge^j Q$.

\begin{cor}[\cite{GMSZ1}]\label{cor:QK_Coulomb}
	In $\qk(X)$, 
	\begin{align}\label{eq:QK-Whitney-Coulomb}
			\wedge_y S\bullet \wedge_y  Q &= (1+y)^N - qy^k(\wedge_yS-1)\cdot \det  Q.
	\end{align}
\end{cor}
	\begin{proof}
	
	For $j<k$ and any $i$,
	$$\wedge^i S\bullet \wedge^j Q = \kappa(\wedge^i \mathrm{S}\cdot \mathrm{H}_j) = \wedge^i S \cdot \wedge^j Q.
	$$
	Here the last equality is the consequence of Proposition~\ref{prop:kappa-identity} and observing that $\wedge^i \mathrm{S}\cdot \mathrm{H}_j\in\ca{R}^{r,j+1}\subset \ca{R}^{r,k}$.
	When $j=k$, Corollary~\ref{cor:quant_prod_Wedge^i_and_detQ} implies 
	$$\wedge_y S\bullet \det  Q = \wedge_y S\cdot  \det  Q  - q(\wedge_yS-1)\cdot \det  Q. $$
	Combining the above two cases, we obtain the required identity \eqref{eq:QK-Whitney-Coulomb}.
	\end{proof}
\begin{cor}[\cite{GMSZ2}] 
	In $\qk(X)$, the `quantum $K$-Whitney' relations are satisfied:
	\begin{align*}
		\wedge_y S\bullet \wedge_y  Q=(1+y)^N-\frac{q}{1-q}y^{k}(\wedge_y S - 1)\bullet \det  Q.
	\end{align*}
\end{cor}
\begin{proof}
		Corollary~\ref{cor:quant_prod_Wedge^i_and_detQ} is equivalent to \cite[Theorem 1.3]{GMSZ2}, and it states that $(\wedge_yS-1)\bullet \det Q =(1-q)(\wedge_yS-1)\cdot \det Q$. The result then follows from Corollary \ref{cor:QK_Coulomb}. 
\end{proof}

\subsection{Euler characteristic}

Let $\chi^q:\qk(X)\rightarrow \Z[[q]]$ denote the $\Z[[q]]$-linear extension of the Euler characteristic
that sends each Schubert structure sheaf $\O_\lambda$ to $\chi(X,\O_\lambda)=1$.

\begin{prop}
	Let $\mathrm{V}\in R(\GLr)\cong K(\fr{X})$ be the associated $K$-theory class.
	We have
	\[
	\langle\!\langle
	\mathrm{V}
	\rangle\!\rangle^{\quot}_0
	=
	\frac{
		1
	}
	{1-q}	\chi^{q}
	\left(X,
	\kappa(\mathrm{V})
	\right)
	.
	\]
	
\end{prop}
\begin{proof}
	It folows from Lemma~\ref{lem:kappa-trivial} that
	$
	\langle\!\langle
	\mathrm{V}
	\rangle\!\rangle^{\quot}_0
	=
	\langle\!\langle
	\kappa(V)
	\rangle\!\rangle^{\quot}_0
	.	
	$
	If $\lambda\in \P_{r,k}$, then
	\[
	\langle\!\langle
	\mathrm{O}_\lambda
	\rangle\!\rangle
	^{\quot}_0
	=\frac{1}{1-q}
	=\frac{1}{1-q} \chi(X,\O_\lambda).
	\]
	We conclude the proof by noting that $\kappa(V)$ is a linear combination of $\mathrm{O}_{\lambda}, \lambda\in\P_{r,k}$
	with coefficients in $\bb{Z}[[q]]$.	
\end{proof}

\begin{rem}
	Consider the image of the quantized dual $\mathrm{O}_\lambda^*$ under the quantum reduction map
	$
	\kappa(\mathrm{O}_\lambda^*)
	=	\kappa(
	\det(\mathrm{S})
	\cdot
	\mathrm{O}_{\lambda^*}
	)
	=
	\det(S)\bullet
	\O_{\lambda^*}.
	$
	It was observed in~\cite[Theorem 5.15]{Buch-Mihalcea} that
	\[
	\frac{\det(S)}{1-q}\bullet\O_{\nu^*}
	\]
	is the $\Z[[q]]$-linear dual to $\O_\lambda$ with respect to $\chi^q$. Indeed,
	\begin{align*}
		\frac{1}{1-q}\chi^q(
		\O_\lambda\bullet \det(S)\bullet \O_{\nu^*}
		)
		&=\frac{1}{1-q}\chi^q\big(
		\kappa(
		\mathrm{O}_\lambda
		\cdot \mathrm{O}^*_{\nu})
		\big)=
		\langle\!\langle
		\mathrm{O}_\lambda
		\cdot
		\mathrm{O}_{\nu}^*
		\rangle\!\rangle^{\quot}_0 =\delta_{\lambda,\nu}.
	\end{align*}
	This can be viewed as a manifestation of the orthogonality relation in Proposition~\ref{prop:quantized-dual-basis}.
\end{rem}

\section{Quasimap wall-crossing}\label{sec:Quasimap wall-crossing}
	In this section, we prove that the quasimap invariants are independent of the stability parameter $\epsilon$. The proof relies on the $K$-theoretic $\epsilon$-wall-crossing formula in \cite[Theorem 1.3]{ZZ1}. In particular, we show that the correction term vanishes by calculating a certain $I$-function.

\subsection{$K$-theoretic $\epsilon$-stable quasimap invariants}

\begin{definition}
An $n$-pointed, genus-zero quasimap of degree $d$ to $X=\grass$ is a tuple
$$(
(C,\bp), E, s
)$$
where 
\begin{enumerate}[\normalfont(i)]
\item $(C,\bp)=(C,p_1,\dots,p_n)$ is a prestable $n$-pointed curve, i.e., a connected, at most nodal, projective curve of genus-zero, together with $n$ distinct and nonsingular markings on it.
\item $E$ is a vector bundle on $C$ of rank $r$ and degree $d$.
\item $s\in H^0(E\otimes\ca{O}^{\oplus N}_{C})$ represents $N$ sections which generically generate $E$.
\end{enumerate}
\end{definition}
A point $p\in C$ is called a \emph{base point} if $s$ does not span the fiber of $E$ at $p$. For any $\epsilon\in\Q_{>0}$, let $\epqm$ be the moduli stack of $\epsilon$-stable quasimaps to $X$ introduced in~\cite{Toda,CKM}.
It parametrizes tuples $((C,\bp),E,s)$ of quasimaps satisfying the following
\begin{enumerate}
\item The base points of $s$ are away from the nodes and the marked points.

\item $\det(E)^\epsilon\otimes \omega_{C}(\sum_{i=1}^{n} p_i)$ is ample,
\item for any point $p\in C$, the torsion subsheaf $\tau(Q)$ of $Q:=\mathrm{coker}(E^{\vee}\xrightarrow{s^{\vee}}\O_{C}^{\oplus N})$ satisfies
\[
\epsilon\cdot \operatorname{length}\tau (Q)_{p}\leq 1.
\]
\end{enumerate}

According to Theorem 1.1 and Lemma 2.15 of~\cite{Toda}, the moduli stack $\epqm$ is a smooth proper DM stack over $\C$.
It is empty unless $-2+n+\epsilon d>0$.
For a fixed degree $d$, the space $\Q_{>0}$ is divided into stability chambers by finitely many walls $\{1/d' \mid d\in \Z_{>0},d'\leq d\}$.
There are two extreme chambers $(2,\infty)$ and $(0,1/d]$, in which the stability conditions are denoted by
$\epsilon=\infty$ and $\epsilon=0+$, respectively.
In these two extreme chambers, we recover the moduli spaces of stable maps and \emph{stable quotients}, respectively, i.e., we have
\begin{align}\label{eq:stable_quotients}
\epqm &\cong \ms,\quad \epsilon >2,\nonumber
\\
\epqm & \cong  \stableq,
\quad
0<\epsilon \leq 1/d.
\end{align}
Here $\stableq$ is the moduli stack of stable quotients introduced in~\cite{MOP}.

Let
$
\ev_i:\epqm\rightarrow X,1 \leq i\leq n$ be the evaluation maps. The diagonal $T$-action on $\O_{ C}^{\oplus N}$ induces a $T$-action on $\epqm$.
\begin{definition}
We define the $T$-equivariant $K$-theoretic $\epsilon$-stable quasimap invariant with insertions $E_{1},\dots,E_{n}\in \kg(X)$ by
\[
\langle
E_{1},
\dots,
E_{n}
\rangle^{\epsilon,T}_{0,n, d}
:=\chi^{T}
\bigg(
\epqm,
\prod_{i=1}^{n}\ev_{i}^{*}(E_{i})
\bigg).
\]
\end{definition}

\subsection{$\epsilon$-wall-crossing}

\begin{theorem}
\label
{prop:epsilon-wall-crossing}
Suppose $n\geq 3$ or $n=2$ and $d>0$.
Then the $T$-equivariant $K$-theoretic $\epsilon$-stable quasimap invariants are independent of $\epsilon$, i.e.,
\[
\langle
E_{1},
\dots,
E_{n}
\rangle^{\epsilon,T}_{0,n, d}
=
\langle
E_{1},
\dots,
E_{n}
\rangle^{\epsilon',T}_{0,n, d}
\]
for any $\epsilon,\epsilon'\in\Q_{>0}$.
\end{theorem}
\begin{proof}
Let $\bt$ be a generic element in $\kg(X)$.
There is a natural $S_{n}$-action on $\epqm$ defined by permuting the $n$ marked points. Here $S_{n}$ denotes
the symmetric group on $n$ letters.
Consider the following (virtual) $(T\times S_{n})$-module
\[
[\bt,\dots, \bt]^{\epsilon,T
}_{0,n,d}
:=
\sum_{j\geq 0}
(-1)^{j}H^{j}
\bigg(\epqm,
\prod_{i=1}^{n} \ev_{i}^{*}(\bt)
\bigg).
\]

Let $\epsilon=1/d_{0}$ be a wall, where $d_{0}\leq d$. Let $\epsilon_{-}<\epsilon_{+}$ be the stability conditions in the two adjacent chambers separated by $\epsilon_{0}$.
Set $m=\lfloor d/d_{0}\rfloor$.
By the $K$-theoretic $\epsilon$-wall-crossing formula~\cite[Theorem 1.3]{ZZ1}\footnote{~\cite[Theorem 1.3]{ZZ1} is stated in the non-torus-equivariant setting, but it holds true in general and the proof is verbatim. The proof uses the master space constructed in \cite{Zhou} for quasimap wall-crossing for cohomological invariants.},
we have 
\begin{equation}
\label{eq:single-epsilon-wall-crossing}
[\bt,\dots, \bt]^{\epsilon_{-},T}_{0,n,d}
-
[\bt,\dots, \bt]^{\epsilon_{+},T}_{0,n,d}
=
\sum_{j=1}^{m}
[\bt,\dots, \bt,
\mu_{d_{0}}(L),\dots,\mu_{d_{0}}(L)
]_{0,n+j,d-jd_{0}}^{S_j,\epsilon_{+},T}
\end{equation}
as virtual $(T\times S_{n})$-modules.
Here
	\[
[\bt,\dots, \bt,
\mu_{d_{0}}(L),\dots,\mu_{d_{0}}(L)
]_{0,n+j,d-jd_{0}}^{S_j,\epsilon_{+},T}
:=
p_{*}
\bigg(
\prod_{i=1}^{n}
\ev_{i}^{*}(\bt)
\cdot
\prod_{a=1}^{j}
\ev_{n+a}^{*}
(\mu_{d_{0}}(\omega)
)|_{\omega=L_{n+a}}
\bigg),
\]
where $\mu_{d_{0}}(\omega)$ is a Laurent polynomial in $\omega$ defined by~\eqref{eq:I-function-truncation} with coefficient in  $\kg(X)$, $L_{n+a}$ is the cotangent line bundle at the $(n+a)$-th marked point, $\ev^*_{n+a}(\mu_{d_0}(\omega))$ is then specialized by setting the formal variable ${w=L_{n+a}}$ to obtain an equivariant $K$-theory class on $Q^{\epsilon_{+}}_{0,n+j}(X,d-jd_{0})$, and
\[
p: [Q^{\epsilon_{+}}_{0,n+j}(X,d-jd_{0})/(T\times S_{j}))]\rightarrow BT.
\]
By Lemma~\ref{lem:laurent-part}, we have $\mu_{d}(\omega)=0$, and therefore the right-hand side of~\eqref{eq:single-epsilon-wall-crossing} vanishes.
Hence $[
	\bt,
	\dots,
	\bt
	]^{\epsilon,T}_{0,n, d}$ is independent of $\epsilon$.

Now set $\bt=\sum_{i=1}^n \alpha_iE_i$ where $\alpha_i$'s are formal variables. Then the coefficient of $\prod_{i=1}^{n}\alpha_{i}$ in the character of the virtual $T$-module $[\bt,\dots,\bt]^{\epsilon,T}_{0,n, d}$ is given by 
\[
	n!\cdot \langle
E_1,\dots, E_n
\rangle^{\epsilon,T}_{0,n,d}\in \Gamma.
\]
This concludes the proof.
\end{proof}

\subsection{$K$-theoretic $I$-function and its truncations}
\label
{subsec:I-function}

To define the Laurent polynomial $\mu_{d}(\omega)$, we refer to the quasimap graph space
$QG^{\epsilon=0+}_{0,n}(X,d)$ as introduced in~\cite[\textsection{7.2}]{CKM},
along with the $\mathbb{C}^*$-action on $QG^{\epsilon=0+}_{0,n}(X,d)$
induced from a $\mathbb{C}^*$-action on $\bP^1$.
The relevant calculations of the fixed loci are detailed in \cite{bck} and \cite[Appendix]{RZ2}. 
We will summarize these calculations to prove the vanishing of $\mu_{d}(\omega)$
as outlined in Lemma \ref{lem:laurent-part}.

\begin{definition}
The $(\epsilon=0+)$-stable quasimap graph space, denoted $QG^{\epsilon=0+}_{0,n}(X,d)$, parameterizes tuples
$$(
(C,\bp), E, s,\varphi
)$$
where $\varphi:C\to \bP^1$ is a degree one map, $((C,\bp), E, s)$ is an $n$-pointed genus-zero quasimap of degree $d$ to $X$, and it satisfies the following stability conditions:
	\begin{enumerate}[\normalfont(i)]
	\item The base points of $s$ are away from the nodes and the marked points.
	\item The line bundle 
	\[
	\det(E)^\epsilon\otimes \omega_{C}\big(\sum_{i=1}^{n} p_i\big)
	\otimes \varphi^*(\omega_{\bb{P}^{1}}^{-1}\otimes M),
	\]
	is ample for every sufficiently small rational numbers $0<\epsilon<\!\!<1$, where $M$ is any ample line bundle on $\bP^{1}$.
\end{enumerate}
\end{definition}

Let $\pi:\ca{C}\rightarrow QG^{\epsilon=0+}_{0,n}(X,d)$ be the universal curve and let 
\[
0\rightarrow
\ca{S}
\rightarrow
\ca{O}^{\oplus N}_{\ca{C}}
\rightarrow
\ca{Q}
\rightarrow 
0
\]
be the universal exact sequence over $\ca{C}$. Here $\ca{S}=\ca{E}^\vee$, where $\ca{E}$ is the universal vector bundle. The higher direct image
$R^{1}\pi_{*}
\left(
\Hom(\ca{S},\ca{Q})
\right)
$
vanishes because over each component of a geometric fiber $C$ of $\pi$,
the bundle $\Hom(\ca{S},\ca{Q})|_{C}$ splits into a direct sum of line bundles of nonnegative degrees.
Hence $\pi_{*}
\left(
\Hom(\ca{S},\ca{Q})
\right)
$
is a vector bundle on $QG^{\epsilon=0+}_{0,n}(X,d)$.
Let $\Phi:
\ca{C}\rightarrow \bP^{1}$ be the universal parametrization map and let $\ca{D}_{i}\subset \ca{C}$ be the divisor corresponding to the $i$-th marked point.
According to the deformation theory of quasimaps
(see, e.g.,~\cite[\textsection{6.2}]{RZ2},~\cite[\textsection{3.2}]{MOP},~\cite[\textsection{2.4}]{Toda}),
the moduli stack $QG^{\epsilon=0+}_{0,n}(X,d)$ is smooth, and its tangent bundle fits into the following exact sequence
\begin{multline*}
0
\rightarrow
\pi_{*}
\bigg(\Omega_{\pi}^{\vee}
\bigg (-\sum_{i=1}^{n}
\ca{D}_{i}
\bigg)
\bigg)
\rightarrow
\pi_{*}
\left(
\Hom(\ca{S},\ca{Q})
\right)
\oplus 
\pi_{*}
(
\Phi^{*}
(
T_{\bP^{1}}
)
)
\\
\rightarrow
T_{QG^{\epsilon=0+}_{0,n}(X,d)}
\rightarrow
R^{1}\pi_{*}
\bigg(\Omega_{\pi}^{\vee}
\bigg (-\sum_{i=1}^{n}
\ca{D}_{i}
\bigg)
\bigg)
\rightarrow
0.
\end{multline*}

Let $T=(\C^{*})^{N}$ be the diagonal torus acting on $QG^{\epsilon=0+}_{0,n}(X,d)$.
Consider the $\C^{*}$-action on $\bP^{1}$ defined by
\[
t[\zeta_{0},\zeta_{1}]
=
[t\zeta_{0},\zeta_{1}],
\quad t\in\C^{*}.
\]
This induces a $\C^{*}$-action on $\QGS$
which commutes with the $T$-action.
Let $\omega$ denote the character of this $\mathbb{C}^*$-action. We denote the unique marking by $x_{\bullet}$.
Set $0:=[1:0]$ and $\infty:=[0:1]$.
Let $F_{d}\subset \QGS$ be the distinguished fixed-point component in $\QGS$ parametrizing $\C^{*}$-fixed quasimaps such that $x_{\bullet}$ lies at $\infty$ and the only base point of $s$ is at $0$.

The graph space $\QGS$ and its distinguished $\C^{*}$-fixed-point component $F_{d}$ are both smooth. The normal bundle is given by the moving part (i.e., the part with non-zero weight) of the restriction of the tangent bundle $T_{QG^{\epsilon=0+}_{0,n}(X,d)}$ to the fixed locus $F_d$, and it decomposes into two parts:
	\[
	N := N_{F_d/\QGS} = L \oplus N^{\mathrm{rel}}.
	\]
\begin{enumerate}[\normalfont(i)]
	\item $L$ is a rank-1 part coming from deforming $x_{\bullet}$ away from $\infty\in\bP^{1}$. Note that $L$ is a trivial line bundle, with equivariant $K$-theoretic Euler class equal to pullback of the tangent space at $\infty\in\bP^1$, which is given by $1 - \omega^{-1}$;
	\item $N^{\rm rel}:=N^{\rm rel}_{F_{d}/\QGS}$ is the moving part of the vector bundle $\pi_{*}
	\left(
	\Hom(\ca{S},\ca{Q})
	\right)
	$.
\end{enumerate}

Let $\ev_{\bullet}:\QGS\rightarrow X\times \bP^{1}$ be the evaluation map at the marking $x_{\bullet}$.
Define $\widetilde{\ev}_{\bullet}:=\pi_{X}\circ \ev_{\bullet}$, where $\pi_{X}:X\times \bP^{1}\rightarrow X$ is the projection.
For $d>0$, we recall the definition of the $I$-function, see for example \cite{ZZ1},
\[
\tilde{I}_{d}(\omega)
:=
(1-\omega)(1-\omega^{-1})
(\wev_{\bullet})_*
\bigg(\frac{
	1
}
{\wedge_{-1}^{T\times\bb{C}^*}(N^\vee)}\bigg)
=
(1-\omega)
(\wev_{\bullet})_*
\bigg(\frac{
	1
}
{\wedge_{-1}^{T\times\bb{C}^*}((N^{\rm rel})^\vee)}\bigg).
\] 
This is a rational function in $\omega$ with coefficients in $\kg(X)$. Any rational function $f(x)$ admits the partial fraction decomposition of the form $f(x)=h(x)+A(x)/B(x)$, where $h(x)$ is a Laurent polynomial part, and $A,B$ are polynomials with $\deg A<\deg B$ and $B(0)\ne 0$. Let 
	\begin{equation}
	\label{eq:I-function-truncation}
	\mu_{d}(\omega)
	:=
	\big[
	\tilde{I}_{d}(\omega)
	\big]_{+}
	\end{equation}
	denote the Laurent polynomial part in the partial fraction decomposition of the rational function $\tilde{I}_{d}(\omega)$.
	
The distinguished $\C^{*}$-fixed-point component $F_{d}$ is explicitly identified in~\cite{bck} (see also \cite{RZ2}).
The irreducible components of $F_{d}$ are indexed by partitions $(d):=(d_1,d_{2},\dots, d_{r})$ of $d$ satisfying
\begin{equation*}
\label{eq:partition-d}
\sum d_i=d\quad\text{and}\quad0\leq d_1\leq d_2\leq\dots\leq d_r.
\end{equation*}
Let $\{r_i\}_{1\leq i\leq j}$ be the jumping index of $\{d_i\}$, i.e.,
$0\leq d_1=\dots=d_{r_1}<d_{r_1+1}=\dots=d_{r_2}<\dots<d_{r_{j}+1}=\dots=d_{r}$.
The (restriction of the) evaluation map $\widetilde{\ev}_{\bullet}: F_{d}\rightarrow X$ can be identified with the flag bundle map
\[
\rho:
{\rm Fl}(r_{1},\dots, r_{j};S)
\rightarrow X,
\]
where ${\rm Fl}(r_{1},\dots, r_{j};S)
$ denote the relative flag variety of type $\{r_{i}\}$.
Let
\[
0\subset S_{r_1}\subset S_{r_2}\subset\dots\subset  S_{r_j}\subset S_{r_{j+1}} =\rho^*S
\]
be the universal flag on ${\rm Fl}(r_{1},\dots, r_{j};S)
$. 
Using the splitting principle, we write 
\[
\sum_{s=1}^{r_a-r_{a-1}}\ca{L}_{r_{a-1}+s}=(S_{r_a}/S_{r_{a-1}})^\vee.
\]
Set $\bar{r}_a:=r_a-r_{a-1}$ and $d_{ba}:=d_{r_b}-d_{r_a}$ with $r_{0}=0$.

According to the computations detailed in~\cite[Appendix A]{RZ2}, which build on the analysis in~\cite{bck}
,
the following explicit formula is derived:
\begin{equation}
\label
{eq:I-d-q}
\tilde{I}_{d}(\omega)
=\sum_{(d)}
\sum_{\sigma
}
(1-\omega)
\cdot
\sigma(I_{(d)})
\end{equation}
where the permutation $\sigma$ runs through all elements in
$S_r/(S_{r_1}\times\dots\times S_{r_{j+1}})$ and acts on the indices of $\ca{L}_{i}$'s in 
\begin{equation*}
\begin{aligned}
I_{(d)}
:=
&\prod_{1\leq a<b\leq j+1}
\prod_{\substack{1\leq s\leq \bar{r}_a\\1\leq t\leq \bar{r}_b}}
\prod_{c=1}^{d_{ba}-1}\ca{L}^\vee_{r_{b-1}+t}\ca{L}_{r_{a-1}+s}
\omega^{c}\\
&\cdot
\prod_{1\leq a<b\leq j+1}
\prod_{\substack{1\leq s\leq \bar{r}_a\\1\leq t\leq \bar{r}_b}}
(-1)^{\bar{r}_a\bar{r}_b(d_{ba}-1)}
\frac{1-\ca{L}_{r_{b-1}+t}^\vee\ca{L}_{r_{a-1}+s}
	\omega^{d_{ba}}}
{1-\ca{L}_{r_{b-1}+t}^\vee\ca{L}_{r_{a-1}+s}}
\\
&\cdot
\frac{
	1}
{\prod_{a=1}^{j+1}\prod_{s=1}^{\bar{r}_a}\prod_{b=1}^{d_{r_a}}\prod_{e=1}^{N}
	(1-\ca{L}^\vee_{r_{a-1}+s}\alpha_{e}^{-1}\omega^b)}.
\end{aligned}
\end{equation*}
The substitution
of $\prod_{k=1}^{N}
(1-\ca{L}^\vee_{r_{a-1}+s}\alpha_{k}^{-1}q^b)$ in place of $
(1-\ca{L}^\vee_{r_{a-1}+s}q^b)^{N}$ in the non-torus-equivariant $I$-function on~\cite[p. 300]{RZ2} corresponds to taking into account the $T$-action on the trivial bundle $\C^{N}$ in the Euler sequence on page 298 of~\cite{RZ2}.

\begin{lemma}
	\label{lem:laurent-part}
	The rational function $\tilde{I}_{d}(\omega)$ is
	regular at $\omega=0$ and vanishes at $\omega=\infty$.
	As a consequence, we have $\mu_{d}(\omega)=
	[
	\tilde{I}_{d}(\omega)
	]_{+}=0$.
	
\end{lemma}
\begin{proof}
	The explicit description of \eqref{eq:I-function-truncation} above immediately implies that $\tilde{I}_d(\omega)$ is regular at $\omega = 0$. To show that the rational function vanishes at $\omega=\infty$, we estimate the degrees of the numerator and denominator of the terms in~\eqref{eq:I-d-q}.
	
	We fix a partition $(d)$ of $d$ and a permutation $\sigma$. Since $k>0$,
	the degree of the numerator of $(1-\omega)
	\cdot
	\sigma(I_{(d)})$ in $\omega$ is bounded by
	\begin{equation}
	\label{eq:numerator}
	1+\sum_{1\leq a<b\leq j+1}
	\frac{\bar{r}_a\bar{r}_b(d_{ba}-1)d_{ba}}{2}
	+\sum_{1\leq a<b\leq j+1}\bar{r}_a\bar{r}_bd_{ba}
	=1+\sum_{1\leq a<b\leq j+1}
	\frac{\bar{r}_a\bar{r}_b(d_{ba}+1)d_{ba}}{2}
	.
	\end{equation}
	Note that the first summation is absent in the case $r=1$. The degree of the denominators of~\eqref{eq:I-d-q} in $\omega$ is 
	\begin{equation}
	\label{eq:denominator}
	\sum_{a=1}^{j+1}
	\frac{\bar{r}_a(d_{r_a}+1)d_{r_a}N}{2}.
	\end{equation}
	By using
	$
	\sum_{a<b}\bar{r}_{a}\leq r-1
	$
	and
	$0<d_{ba}\leq d_{r_{b}}$,
	we obtain the following upper bound for the first term of the RHS of~\eqref{eq:numerator}:
	\begin{align*}
	\sum_{1\leq a<b\leq j+1}
	\frac{\bar{r}_a\bar{r}_b(d_{ba}+1)d_{ba}}{2}
	\leq
	\sum_{1\leq a<b\leq j+1}
	\frac{\bar{r}_{a}\bar{r}_{b}(d_{r_{b}}+1)d_{r_{b}}}{2}
	\leq
	\frac{r-1}{2}\sum_{b=1}^{j+1}\bar{r}_{b}(d_{r_{b}}+1)d_{r_{b}}.
	\end{align*}
	It follows that the difference $\eqref{eq:denominator}-\eqref{eq:numerator}$ is greater than or equal to
	\begin{equation}
	\label{eq:last-lower-bound}
	\frac{N-(r-1)}
	{2}
	\sum_{a=1}^{j+1}\bar{r}_a
	(d_{r_a}+1)d_{r_a}
	-1>0.
	\end{equation}

\end{proof}

\section{Light-to-heavy wall crossing}\label{sec:Light-to-heavy wall crossing}
In this section, we derive a light-to-heavy wall-crossing
result for $K$-theoretic quasimap invariants in all genera using the
tools from~\cite[\textsection{6}]{RZ2}.
We then apply this result specifically to the genus-zero case to complete the proof of Theorem~\ref{thm:intro-Wall_crosing}.
\subsection{$K$-theoretic $(0+)$-stable graph quasimap invariants}\label{sec:stable graph quasimap invariants}
We fix a smooth curve $(C,\mathbf{q})=(C,q_1,\dots,q_n)$ of genus $g$ with $n$ distinct markings.
In this subsection, we consider quasimaps with one parametrized component isomorphic to $C$ and their moduli space.

\begin{definition}
An $n$-pointed, genus $g$ quasimap with one parametrized component of degree $d$ to $X=\grass$ is a tuple
$$(
(\widehat{C},\bp), E, s,\varphi
)$$
where 
\begin{enumerate}[\normalfont(i)]
\item $((\hc,\bp), E, s)$ is an $n$-pointed, genus $g$ quasimap of degree $d$ to $X$.
\item $\varphi: \hc\rightarrow C$ is a morphism of degree one such that 
\begin{equation}
\label{eq:bubbling}
\varphi(p_{i})=q_{i},
\quad1\leq i\leq n
.
\end{equation}
\end{enumerate}

\end{definition}

By definition, $\hc$ has an irreducible component $\hc_0$ such that $\varphi|_{\hc_0}:\hc_0\xrightarrow{\sim} C$ is an isomorphism, and the rest of $\hc$ is contracted by $\varphi$.
We refer to $\hc_{0}$ as the parametrized component.
An isomorphism between two quasimaps $((\hc,\bp), E, s,\varphi)$ and $((\hc',\bp'), E', s',\varphi')$ consists of an isomorphism $f:\hc\xrightarrow{\sim} \hc'$ of the underlying curves, along with an isomorphism $\sigma:E\rightarrow f^*E'$, such that
\[
f(p_i)=p'_i,\ \varphi=\varphi'\circ f,\,\text{and}\ \sigma(s)=f^*(s').
\]

\begin{definition}
\label{def:stability-quasimaps}
A quasimap $((\hc,\bp), E, s,\varphi )$ with one parametrized component is $(0+)$-stable if the following two conditions hold:
\begin{enumerate}[\normalfont(i)]
\item The base points of $s$ are away from the nodes and the markings.
\item The line bundle 
\[
\det(E)^\epsilon\otimes \omega_{\hc}\big(\sum_{i=1}^{n} p_i\big)
\otimes \varphi^*(\omega_{C}^{-1}\otimes M),
\]
is ample for every sufficiently small rational numbers $0<\epsilon<\!\!<1$, where $M$ is any ample line bundle on $C$.
\end{enumerate}
\end{definition}
For a $(0+)$-stable tuple $((\hc,\bp), E, s,\varphi )$, the source curve $\hc$ is obtained from the parametrized component $\hc_{0}$ by adding to $q_{i}$, for all $i$, a (possibly empty) chain of rational curves with $p_{i}$ on the extremal components,
and $(E,s)$ induces a rational map from $\hc$ to $X$ that is regular at the nodes and the markings $p_{i}$.
We refer the reader to~\cite[Fig. 2]{RZ2} for an illustration of the source curve $\hc$.

Let $\qgs$ denote the moduli stack of $(0+)$-stable quasimaps with one parametrized component
as defined in Definition~\ref{def:stability-quasimaps}.
It forms a closed substack of the moduli stack $\mathrm{Qmap}_{g,n}(X,d;C)$ defined in~\cite[\textsection{7.2}]{CKM},
which parametrizes the data from Definition~\ref{def:stability-quasimaps} without the
condition~\eqref{eq:bubbling}.
According to~\cite[Theorem 7.2.2]{CKM}, $\mathrm{Qmap}_{g,n}(X,d;C)$ is a proper DM stack,
implying that $\qgs$ is also a proper DM stack.

Let $\ca{S}$ and $\ca{Q}$ be the universal subbundle and quotient sheaf on the universal curve $\pi:\ca{C}\ra\qgs$.
The complex $R\shom_\pi (\ca{S},\ca{Q})$ defines a ($T$-equivariant) POT for $\nu:\qgs\ra\ca{T}_{\rm maps}(C,\mathbf{q})$,
where $\ca{T}_{\rm maps}(C,\mathbf{q})$ denotes the smooth Artin stack parametrizing
the data $((\widehat{C},\bp),\varphi))$ satisfying~\ref{eq:bubbling}.
Define
\[
\ovir_{\qgs}:=\nu^!(\ca{O}_{\ca{T}_{\rm maps}(C,\mathbf{q})})\in \kG(\qgs).
\]
Let $
\ev_i:\qgs\rightarrow X,
1 \leq i\leq n
$ be the evaluation morphism.

\begin{definition}
Let $E_{1},\dots,E_{n}\in \kg(X)$.
We define the
$T$-equivariant $K$-theoretic $(0+)$-stable graph quasimap invariant with the insertions $E_{i}$'s by
\[
\langle
E_{1},
\dots,
E_{n}
\rangle^{0+,T}_{g,d}
:=\chi^{\vir,T}
\bigg(
\qgs,
\prod_{i=1}^{n}\ev_{i}^{*}(E_{i})
\bigg).
\]
\end{definition}

\subsection{Light-to-heavy wall-crossing}

In this subsection, we use the setup from~\cite[\textsection{6.2}]{RZ2} to derive a light-to-heavy wall-crossing
formula
that connects the $K$-theoretic Quot scheme invariants
with the $K$-theoretic $(0+)$-stable graph quasimap invariants. Let $m$ be an integer such that $1\leq m\leq n$. We relabel the first $m$ markings as $x_1,\dots, x_m$ and the last $n-m$ markings as $y_1,\dots,y_{n-m}$.
Set $\bx=(x_1,\dots,x_m)$ and $\by=(y_1,\dots,y_{n-m})$.
\begin{definition}
	\label{def6.7}
A quasimap $((\hc,\bx;\by), E, s,\varphi)$
with one parametrized component to $X$ is called $(0+,0+)$-stable with $(m,n-m)$-weighted markings and degree $d$ if :
\begin{enumerate}[\normalfont(i)]
\item The quasimap only has finitely many base points which are away from the nodes.
\item $x_1,\dots,x_m$ are not base points of $s$.
\item The $\bb{Q}$-line bundle 
\[
\det(E)^\epsilon\otimes \omega_{\hc}\bigg(\sum_{i=1}^{m} x_i+\epsilon'\sum_{j=1}^{n-m} y_j\bigg)\otimes \varphi^*(\omega_{C}^{-1}\otimes M)
\]
is relatively ample for every sufficiently small rational numbers $0<\epsilon,\epsilon'<\!\!<1$.
Here $M$ is any ample line bundle on $C$.
\item $\rank(E)=r$ and $\deg(E)=d$.
\end{enumerate}
\end{definition}
We refer to $x_i$ as the \emph{heavy} markings and $y_{i}$ as the \emph{light} markings.
Note that each $x_i$ is connected to the parametrized component $\hc_0\cong C$
through a (possibly empty) chain of $\bP^1$'s, and $y_i$ is located on $\hc_0$ and corresponds to $q_{m+i}$.
Let $\msmm$ denote the moduli stack of $(0+,0+)$-stable quasimaps with one parametrized component to $X$ with $(m,n-m)$-weighted markings.
It follows from the definition that
\begin{equation}
\label{eq:two-extreme-chambers}
Q^{0+,0+}_{C,0\mid n}(X,d)\cong\QuotC,\quad
Q^{0+,0+}_{C,n \mid 0}(X,d)=\qgs.
\end{equation}

Similar to $\qgs$,
the moduli stack $\msmm$
is a proper DM stack
and equipped with a ($T$-equivariant) virtual structure sheaf $\ovir_{\msmm}$.
There are two types of evaluation morphisms:
\begin{align*}
\ev_i: \msmm
&
\rightarrow
X,
\quad i=1,\dots,m,\\
\widetilde{\ev}_j:\msmm
&
\rightarrow
\fr{X}=[M_{r\times N}/\GLr],
\quad j=1,\dots, n-m.
\end{align*}

Let $\mathrm{V}_1,\dots,\mathrm{V}_m,\mathrm{W}_1,\dots,\mathrm{W}_{n-m}$ be (virtual) $\GLr$-modules.
We define their associated classes in $\kg(\fr{X})$ by
$
M_{r\times N}\times_{\GLr}\mathrm{V}_{i}$
and 
$M_{r\times N}\times_{\GLr}\mathrm{W}_{j}
$,
which are still denoted by $\mathrm{V_i}$ and $\mathrm{W}_j$ by convention.
Let $\ca{V}_{i}$ and $\ca{W}_{j}$ denote the restrictions of the above classes to $X$.

\begin{definition}
We define the $(m,n-m)$-weighted invariant by \footnote{Here we use calligraphic font for elements in $K_T^{0}(X)$, which are denoted by italics in the rest of the paper.}
\[
\langle \ca{V}_1, \cdots \ca{V}_m|\mathrm{W}_1,\dots,\mathrm{W}_{n-m}\rangle^{0+,T}_{g,d,m \mid n-m}
:=
\chi^{\vir,T}
\Big(
\msmm,
\prod_{i=1}^m \ev_i^*(\ca{V}_i)
\cdot\prod_{j=1}^{n-m} 
\widetilde{\ev}_{j}^*(\mathrm{W}_j) 
\Big).
\]
\end{definition}

The above invariant is a special case of the one in~\cite[Definition 6.12]{RZ2}, specifically when
$e^*=0$ and $l=0$.
Therefore, we can use the same proof to establish the following slight variation of~\cite[Theorem 6.13]{RZ2}:

\begin{proposition}
\label
{prop:wall-crossing}
Suppose that $\mathrm{W}_{1}\in \ca{R}^{r,k}$.
We have
\[
\langle \ca{V}_1, \cdots \ca{V}_m|
\mathrm{W}_1,\dots,\mathrm{W}_{n-m}\rangle^{0+,T}_{g,d,m \mid n-m}
=\langle \ca{V}_1, \cdots \ca{V}_m,\ca{W}_1|
\mathrm{W}_2\dots,\mathrm{W}_{n-m}\rangle^{0+,T}_{g,d,m+1 \mid n-m-1}.
\]
\end{proposition}

\begin{proof}
The master space $\widetilde{\ca{M}}_{C,m\mid n-m}$ introduced in~\cite[\textsection{6}]{RZ2} has commuting $T$- and $\C^{*}$-actions.
By applying the $K$-theoretic localization formula for the $\C^{*}$-action on the master space,
we obtain the $T$-equivariant version of~\cite[(59)]{RZ2} that holds for $T$-equivariant Euler characteristics.
We use the same argument as in the proof of~\cite[Theorem 6.13]{RZ2} to reduce the problem to showing $\mu_{d}^{\mathrm{W}_1}(\omega)=0$ defined in \eqref{eq:I_function_with_W} for any $\mathrm{W}_{1}\in \ca{R}^{r,k}$. We explain this vanishing in Lemma~\ref{lem:vanishin_mu(omega)} below.
\end{proof}
Recall the notations from Section~\ref{subsec:I-function}. Consider $\mathrm{W}=\bb{S}^{\lambda}(\mathrm{S})\in \kg(\fr{X})$ for an integer partition $\lambda\in \P_{r,k}$.
For $d>0$, we define
\begin{equation*} 
I^{\mathrm{W}}_{d}(\omega)
=
(\wev_{\bullet})_*
\bigg(\frac{
	\bS^{\lambda}(
	\ca{S}_{0}
	)
}{\wedge_{-1}^{T\times \bb{C}^*}((N^{\rm rel})^\vee)}\bigg),
\end{equation*}
where $\ca{S}_{0}$ is the restriction of $\ca{S}$ to $F_{d}\times\{0\}$. Let
\begin{equation}\label{eq:I_function_with_W}
\mu_{d}^{\mathrm{W}}(\omega) = [I_{d}^{\mathrm{W}}(\omega)
]_{+}
\end{equation}
be the Laurent polynomial part in the partial fraction decomposition of $I_{d}^{\mathrm{W}}(\omega)$.
For a general $\mathrm{W}\in \ca{R}^{r,k}$, $\mu_{d}^{\mathrm{W}}(\omega)$ can be defined using linearity.

Moreover, the explicit localization calculation in \cite[Appendix A]{RZ2} shows that
\begin{equation*}
I^{\mathrm{W}}_{d}(\omega)
=
\sum_{(d)}
\sum_{\sigma
}
\sigma(
\bS^{\lambda}(
\ca{K}^{(d)}_{0})
\cdot
I_{(d)}),
\end{equation*}
where $I_{(d)}$ is defined in \eqref{eq:I-d-q} and
\[
\ca{K}^{(d)}_{0}=\sum_{a=1}^{j+1}\sum_{s=1}^{\bar{r}_{a}}\ca{L}^{\vee}_{r_{a-1}+s}\omega^{d_{r_{a}}}.
\]

\begin{lem}\label{lem:vanishin_mu(omega)}
	For any $\mathrm{W}\in \ca{R}^{r,k}$, the rational function $I^{\mathrm{W}}_{d}(\omega)
	$ is regular at $\omega=0$ and vanishes at $\omega=\infty$.
	As a consequence, we have $\mu_{d}^{\mathrm{W}}(\omega)=
	[
	I_{d}^{\mathrm{W}}(\omega)
	]_{+}=0$.
\end{lem}
\begin{proof}
	Given $\lambda\in \P_{r,k}$, we consider the $\GLr$-representation $\mathrm{W}=\bS^{\lambda}((\C^{r})^{\vee})$.
	Note that $ \bS^{\lambda}(
	\ca{K}^{(d)}_{0})$ with $\lambda\in\P_{r,k}$
	is a polynomial in $\omega$ whose degree is bounded above by $kd$. The proof is similar to that of Lemma~\ref{lem:laurent-part}, except we have to add $kd$ for the degree of $\bS^{\lambda}(
	\ca{K}^{(d)}_{0})$ in place of $1$ corresponding to the polynomial $(1-\omega)$ in \eqref{eq:I-d-q}.
	In the last step, using the observation:
\[
\sum_{a=1}^{j+1}\bar{r}_a(d_{r_a}+1)d_{r_a}\geq 2 \sum_{a=1}^{j+1}d_{r_a}=2d,
\]
we establish the analogous bound from formula~\eqref{eq:last-lower-bound} as:
\[
\frac{(N-(r-1))}{2}\sum_{a=1}^{j+1}\bar{r}_a(d_{r_a}+1)d_{r_a}-kd\geq (N-(r-1))d-kd
=d
>0.
\]
Thus $I^{W}_{d}(\omega)$ vanishes at $\omega=\infty$.
\end{proof}

By repeatedly applying Proposition~\ref{prop:wall-crossing} and using the isomorphisms~\eqref{eq:two-extreme-chambers}, we obtain  
\begin{corollary}
	\label
	{cor:wall-crossing}
	Suppose that $\mathrm{V}_{1},\dots,\mathrm{V}_{n}\in \ca{R}^{r,k}$. We have 
	\[
	\langle
	\mathrm{V}_{1},
	\dots,
	\mathrm{V}_{n}
	\rangle^{\quot,T}_{g,d}
	=
	\langle
	\ca{V}_{1},
	\dots,
	\ca{V}_{n}
	\rangle^{0+,T}_{g,d}.
	\]
\end{corollary}

\subsection{Proof of Theorem \ref{thm:intro-Wall_crosing}}
In this subsection, we take $C=\bP^1$. Recall the definition of the moduli space of stable quotients $Q^{\epsilon=0+}_{0,n}(X,d)$ from \eqref{eq:stable_quotients}, and the moduli stack of $(0+)$-stable quasimaps with parameterized component $Q^{0+}_{\bP^{1},3}(X,d)$ from Section~\ref{sec:stable graph quasimap invariants}.
\begin{lemma}
\label{lem:identify-3-point-moduli}
There is a natural $T$-equivariant isomorphism 
\[
Q^{\epsilon=0+}_{0,3}(X,d)
\cong
Q^{0+}_{\bP^{1},3}(X,d).
\]
\end{lemma}
\begin{proof}
There is a $T$-equivariant forgetful morphism 
\[
\mu: 
Q^{0+}_{\bP^{1},3}(X,d)
\rightarrow 
Q^{\epsilon=0+}_{0,3}(X,d)
\]
forgetting the parametrization $\varphi$.
Let $U$ be a scheme and let $(\pi:\ca{X}\rightarrow U,p_{1},p_{2},p_{3};\ca{E};s,\varphi)$ be a $U$-point of $
Q^{\epsilon=0+}_{0,3}(X,d)$.
Then the basepoint-free line bundle $\omega_{\pi}(p_{1}+p_{2}+p_{3})$ induces an $U$-morphism $\varphi:\ca{X}\rightarrow
U\times \bP^{1}$ of degree 1.
Hence we obtain an $U$-point of
$Q^{\epsilon=0+}_{0,3}(X,d),
$ 
and therefore a morphism
\[
\nu: 
Q^{\epsilon=0+}_{0,3}(X,d)
\rightarrow
Q^{0+}_{\bP^{1},3}(X,d), 
\]
which is inverse to $\mu$. This finishes the proof of the lemma.

\end{proof}

\begin{theorem}
\label
{prop:QK=quot}
Let $\mathrm{V}_{1},\mathrm{V}_{2},\mathrm{V}_{3}\in  \ca{R}^{r,k}$. Then we have
\[
\langle
\ca{V}_{1},
\ca{V}_{2},
\ca{V}_{3}
\rangle^{T}_{0,3,d}
=
\langle
\mathrm{V}_{1},
\mathrm{V}_{2},
\mathrm{V}_{3}
\rangle^{\quot,T}_{0,d}.
\]
\end{theorem}
\begin{proof}
By Theorem~\ref{prop:epsilon-wall-crossing}, we have
\[
\langle
\ca{V}_{1},
\ca{V}_{2},
\ca{V}_{3}
\rangle^{\epsilon=\infty,T}_{0,3,d}
=
\langle
\ca{V}_{1},
\ca{V}_{2},
\ca{V}_{3}
\rangle^{\epsilon=0+,T}_{0,3,d}
\]
and the former is equal to the quantum $K$-invariant
$
\langle
\ca{V}_{1},
\ca{V}_{2},
\ca{V}_{3}
\rangle^{T}_{0,3,d}
$.
It follows from Lemma~\ref{lem:identify-3-point-moduli} that
\[
\langle
\ca{V}_{1},
\ca{V}_{2},
\ca{V}_{3}
\rangle^{\epsilon=0+,T}_{0,3,d}
=
\langle
\ca{V}_{1},
\ca{V}_{2},
\ca{V}_{3}
\rangle^{0+,T}_{0,d}.
\]
By
Corollary~\ref{cor:wall-crossing}
,
we have
\[
\langle
\ca{V}_{1},
\ca{V}_{2},
\ca{V}_{3}
\rangle^{0+,T}_{0,d}
=
\langle
\mathrm{V}_{1},
\mathrm{V}_{2},
\mathrm{V}_{3}
\rangle^{\quot,T}_{0,d}.
\]
This concludes the proof.
\end{proof}

\begin{rem}
It is well-known that 3-pointed, genus zero Gromov--Witten invariants of Grassmannians with Schubert class insertions
coincide with the corresponding invariants defined using Quot schemes~\cite{Bertram2}.
A proof of this fact can be found in~\cite[\textsection{4}]{CF3}, which utilizes the Kleiman's transversality theorem
and a detailed description of the boundary of the Quot scheme. The wall-crossing strategy employed in this paper uses a different geometric analysis compared to \cite{Bertram2,CF3}.
\end{rem}

\begin{corollary}
\label
{cor:identify-2-point}
Let $V_{1},V_{2}, V\in  \ca{R}^{r,k}$. Then we have
\[
\langle
\ca{V}_{1},
\ca{V}_{2}
\rangle^{T}_{0,2,d}
=
\langle
\mathrm{V}_{1},
\mathrm{V}_{2}
\rangle^{\quot,T}_{0,d},
\quad
\quad
\langle
\ca{V}
\rangle^{T}_{0,1,d}
=
\langle
\mathrm{V}
\rangle^{\quot,T}_{0,d}.
\]
\end{corollary}

\begin{proof}
By Definition~\ref{def:intro-K-theoretic-quot}, we have
\begin{align*}
\langle
\mathrm{V}_{1},
\mathrm{V}_{2},
\O_{\fr{X}}
\rangle^{\quot,T}_{d}
=&
\langle
\mathrm{V}_{1},
\mathrm{V}_{2}
\rangle^{\quot,T}_{d},\\
\langle
\mathrm{V},
\O_{\fr{X}},
\O_{\fr{X}}
\rangle^{\quot,T}_{d}
=&
\langle
\mathrm{V}
\rangle^{\quot,T}_{d},
\end{align*}
where $\O_{\fr{X}}$ is the trivial line bundle on $\fr{X}$.
Similar equalities hold for quantum $K$-invariants due to the string equation~\cite[(22)]{Lee}.
Hence Theorem~\ref{prop:QK=quot} implies the corollary.
\end{proof}

\section{Degeneration formula}
\label
{sec:gluing}
In this section, we adapt the constructions and arguments in~\cite{Li-Wu,Qu}
to prove the degeneration formulas in Theorem~\ref{thm:intro_degeneration_formula} for $K$-theoretic Quot scheme invariants.

\subsection{Expanded degenerations}
We consider a simple degeneration of smooth projective curves.
Let $\ca{C}\rightarrow B$ be a projective morphism from a smooth variety $\ca{C}$ to a smooth pointed curve
$(B,0)$ such that the fibers outside 0 are smooth curves of genus $g$, and the fiber over
0 is a union of two smooth irreducible components $C_-$ and $C_+$ along the node $x_{\rm node}$:
\[
    C_{\rm deg} = C_- \coprod_{x_{\rm node}} C_+.
\]
Let $g_-$ and $g_+$ be the genera of $C_-$ and $C_+$, respectively. Then $g_-+g_+=g$.
Let $x_-$ and $x_+$ be the points on $C_-$ and $C_+$ corresponding to the node $x_{\rm node}$.

Expanded degenerations were introduced in~\cite{Jun1,Jun2}.
We recall the setup in~\cite{Li-Wu}.
Let $0:=[1:0]$ and $\infty :=[0:1]$ be two points on $\bP^1$.
There is a universal family $\fr{C}\rightarrow \fr{B}$ of expanded degenerations associated to
the family $\ca{C}\rightarrow B$,
whose singular fibers are expansions of $C_{\rm deg}$ obtained by inserting a chain of $\bP^1$'s
at the node $x_{\rm node}$.
More precisely, we have
\[
C_{\rm deg}[n]
= C_-\coprod_{x_-=0}
\underbrace{ \bP^{1} \coprod_{\infty=0} \dots \coprod_{\infty=0}\bP^1
 }_\text{$n$ copies} \coprod_{\infty=x_+} C_+,
\]
where $n\geq 0$ and we identify $x_-$ in $C_-$ with 0 in the first copy of $\bP^1$, identify
$\infty$ in the $i$-th $\bP^1$ with $0$ in the $(i+1)$-th $\bP^1$, and identify
$\infty$ in the last $\bP^1$ with $x_+$ in $C_+$. 
We refer to this chain of $\bP^1$'s as a \emph{rational bridge} (of length $n$).
The universal family fits into the following commutative diagram
\[
	\begin{tikzcd}
	\fr{C} \arrow[r
	] \arrow[d] & \ca{C} \arrow[d] \\
	\fr{B} \arrow[r] & B
	\end{tikzcd}
	\]
where the first horizontal map is an isomoprhism on smooth fibers and, on singular fibers, contracts the chains of $\bP^1$'s to
the node $x_{\rm node}$ in $C_{\rm deg}$.
\begin{rem}
For detailed definitions of the stacks $\fr{E}$ and $\fr{B}$,
denoted as $\fX$ and $\fr{C}$ respectively in \cite[\textsection{2.1}, \textsection{2.2}]{Li-Wu},
we direct the reader to that source.
Note that $\fr{B}_0$, the fiber of
$\fr{B}\rightarrow B$ over 0, is independent of $\fr{B}$
and is the same as $\fr{T}_0$ in~\cite{ACFW}.
\end{rem}

For the relative pairs $(C_-,x_-)$ and $(C_+,x_+)$, the universal families of
expanded degenerations (or expanded relative pairs)
are $(\fr{C}_-,\fr{x}_-)\rightarrow \ca{T}$ and $(\fr{C}_+,\fr{x}_+)\rightarrow \ca{T}$,
respectively.
Here, the base $\ca{T}$ will be explained
in Remark~\ref{rem:bases}.
The fibers of $(\fr{C}_\pm,\fr{x}_\pm)\rightarrow \ca{T}$ are expansions
of $(C_\pm, x_\pm)$ obtained by attaching chains of $\bP^1$'s
of various lengths to $C_\pm$ at $x_\pm$:
\begin{align*}
(C_-[n],x_-[n]) &:= C_- \coprod_{x_-=0} 
\underbrace{\bP^1 \coprod_{\infty =0 }\dots \coprod_{\infty =0 }\bP^1}_\text{$n$ copies}, \\
(C_+[n],x_+[n]) &:= 
\underbrace{\bP^1 \coprod_{\infty =0 }\dots \coprod_{\infty =0 }\bP^1}_\text{$n$ copies}\coprod_{\infty = x_+ }C_+, 	
\end{align*}
We refer to these chains of $\bP^1$'s as \emph{rational tails}, and the above expansions are sometimes referred to as \emph{accordions}. We also need the family $(\fr{x}_-, \fr{C}_\sim, \fr{x}_+)\rightarrow \ca{T}_\sim$ of nonrigid expanded
degenerations associated to the node, whose fibers are chains of $\bP^1$'s of length $n$ for $n\geq 0$:
\[
(x_-, \bP^1[n]_\sim, D_+) := 
\underbrace{\bP^1\coprod_{\infty = 0}\dots \coprod_{\infty=0} \bP^1}_\text{$n$ copies}. 	
\]

\begin{rem}
	\label{rem:bases}
The stacks $\ca{T},\ca{T}_\sim$ of expanded degenerations
are independent of the choices of $(C_\pm,x_\pm)$ and the same as those in~\cite{Graber-Vakil, ACFW}.
According to~\cite{Graber-Vakil} (see also~\cite[Proposition 4.2]{ACFW}),
the stack $\ca{T}_\sim$ can be identified with the moduli stack $\fr{M}^{ss}_{0,2}$
of semistable, genus zero curves with two markings, where nodes only separate these markings.
\end{rem}
\begin{rem}
	Our notation $\ca{T}$ for expanded relative pairs aligns with that in~\cite{Qu};
	see Remark 4.3 therein for distinctions from the notation used in~\cite{Li-Wu}.
\end{rem}

\subsection{Good degeneration of Quot schemes}
We specify the definition of \emph{stable admissible quotients} and their moduli space from~\cite{Li-Wu}
to our setting.
According to~\cite[Definition 3.1]{Li-Wu}, for a closed subscheme $D\subset W$ of a noetherian scheme,
a coherent sheaf on $W$ is said to be \emph{normal} to $D$ if $\mathrm{Tor}_1^{\ca{O}_W}(F,\O_D) =0$.
Let $W=C_1 \coprod_p C_2$ be the union of two smooth curves intersecting at the node $p$.
Let $Z$ be a scheme and let $\ca{F}$ be a $Z$-flat family of coherent sheaves on $W\times Z$.
Suppose for any $z\in Z$, the sheaf $\ca{F}_z= \ca{F}\otimes_{\O_Z}\C(z)$ is normal to $p$.
After restricting to an affine neighborhood of $p$, the map (3.2) in~\cite{Li-Wu}
is an isomorphism.
According to~\cite[Proposition 3.5]{Li-Wu}, there is an open neighborhood
$U \subset W\times Z$ of $p\times Z$ such that $\ca{F}|_U$ is flat over $U$.
This implies that $\ca{F}|_U$ is torsion-free and, therefore, locally free at $p\times Z$.

As the family $\fr{C}\rightarrow \fr{B}$ is represented by a projective morphism,
the relative Quot scheme of this family parametrizing rank $N-r$ degree $d$ quotients of the
rank-$N$ trivial sheaf is an algebraic stack projective over $\fr{B}$.
In~\cite{Li-Wu}, the authors constructed an open substack $\QuotCB$
of the relative Quot scheme that parametrizes \emph{stable} quotients.
By~\cite[Definition 4.2]{Li-Wu}, the restriction of a quotient $\phi: \O^{\oplus N}_{\fr{C}} \rightarrow \ca{Q}$
on any smooth fiber of $\fr{C}\rightarrow \fr{B}$ is automatically stable.
Over the component $C_{\rm deg}[n]$ of the singular fibers, the restriction of $\phi$ is said to be stable
if its automorphism group, a subgroup of the automorphism group of $C_{\rm deg}[n]$,
is finite and it is \emph{admissible} in the sense that $\ca{Q}$
is normal to all the nodes in $C_{\rm deg}[n]$.
By~\cite[Lemma 1]{MOP} and the discussion in the previous paragraph,
the stability condition on $\phi$ introduced in~\cite[Definition 4.2]{Li-Wu}
is equivalent to that $\ca{Q}|_{\cdegn}$ is locally free at all the nodes
and the restriction $\ca{Q}|_{\bP^1}$ to any component of the rational bridges is
of positive degree.
Hence, in our setting of expanded degeneration of curves,
we introduce the following equivalent definition:	
\begin{definition}
	\label{def:stable-quotient}
	A $\fr{B}$-flat family of coherent quotients $\ca{O}^{\oplus N}\rightarrow \ca{Q}$
	is stable if the restriction $\ca{Q}|_{\cdegn}$ is locally free at the nodes
	and has positive degrees over all components of the rational bridge in $\cdegn$.
\end{definition}
Let $\QuotCB$ be the moduli space of flat families of stable quotients of rank $N-r$ and degree $d$
over the universal family $\fr{C}\rightarrow\fr{B}$ of expanded degenerations of curves.
It is referred to as the \emph{good degeneration} of Quot schemes in~\cite{Li-Wu}.
We refer the reader to \textsection{4.1} therein for the definition of
stable (admissible) quotients over general expanded degenerations.

Similarly, we define stable relative quotients on $(\fr{C}_\pm,\fr{x}_\pm) \rightarrow \ca{T}$
and $(\fr{x}_-,\fr{C}_\sim,\fr{x}_+)\rightarrow \ca{T}_\sim$
by requiring their restrictions to the fibers to be locally free at the nodes and the markings
$\fr{x}_\pm$ and have positive degrees over all $\bP^1$'s in the rational tails or chains.
Let $\QuotCpm$ and $\QuotCt$
denote the moduli spaces of flat families of stable relative quotients
over the corresponding universal families of expanded degenerations.
According to~\cite[Theorem 4.14, 4.15]{Li-Wu}, all these stacks, including $\QuotCB$
are separate, Deligne--Mumford stacks of finite type;
the first three are proper, whereas
$\QuotCB$ is proper over $B$.
For simplicity, we set
$\ca{M}^d:= \QuotCB$,
$\ca{M}^d_\pm: = \fr{Quot}_{d}(\fr{C}_\pm/\ca{T},N,r)$,
and $\ca{M}^d_\sim:=\QuotCt$.

Now we identify $\ca{M}^d_\pm$ and $\ca{M}^d_\sim$ with the quasimap moduli spaces introduced
in Section~\ref{sec:Light-to-heavy wall crossing}.
By Remark~\ref{rem:bases}, we have $\ca{T}_\sim \cong \fr{M}_{0,2}^{ss}$.
Note that both $\ca{M}^d_\sim$ and the moduli space $Q^{\epsilon=0+}_{0,2}(X,d)$ of
$(\epsilon=0+)$-stable quasimaps, referred to as stable quotients in~\cite{MOP},
are defined over $\ca{T}_\sim$ with matching stability conditions
(see Definition~\ref{def:stable-quotient} and \cite[\textsection{2}]{MOP}).
Hence, we have the isomorphism over $\ca{T}_\sim$:
\begin{equation}
	\label{eq:rational-bridge}
	\ca{M}^d_\sim	\cong
Q^{\epsilon=0+}_{0,2}(X,d).
\end{equation}

Let $C$ be a smooth projective curve of genus $g$.
Let $\fr{M}_{g,1}(C,1)$ denote the stack of maps of degree 1 from 1-pointed prestable curves
of genus $g$ to $C$.
In other words, $\fr{M}_{g,1}(C,1)$ is the moduli stack of genus $g$, 1-pointed
prestable curves with one parametrized component isomorphic to $C$.
We choose a closed point $c\in C$.
Let $\ca{T}_{\rm maps}(C,c)\subset\fr{M}_{g,1}(C,1)$ denote the locally closed substack
parametrizing curves that are obtained from the parametrized component $C$ by
gluing at $c$ a chain of rational curves, with the unique marking at the extremal component
(c.f.~\cite[Definition 3.1.2]{ACFW} in the case $C=\bP^1$).
Let $\ca{T}_{naive}(C,c)$ denote the stack of \emph{naive expansions} of the pair $(C,c)$ (see~\cite[Definition 2.1.3]{ACFW}).
Then by using the arguments in the proof of~\cite[Lemma 3.1.3]{ACFW}, one can show that
\[
\ca{T}_{naive}(C,c)\cong \ca{T}_{\rm maps}(C,c).	
\]
It follows from Lemma 3.2.2 and Proposition 3.2.3 in~\cite{ACFW} that
\[
\ca{T}	\cong \ca{T}_{\rm maps}(C,c).
\]
Let $Q^{0+}_{C_\pm,1}(X,d)$ be the moduli stacks of $(0+)$-stable quasimaps with one
parametrized component as defined in Definition~\ref{def:stability-quasimaps}.
Consider the forgetful morphism $Q^{0+}_{C_\pm,1}(X,d)\rightarrow \ca{T}_{\rm maps}(C_\pm,x_\pm)$
forgetting the stable quotients.
As a consequence of Definition~\ref{def:stable-quotient}, we have an isomorphism over $\ca{T}$:
\begin{equation*}
	\ca{M}_\pm^d \cong Q^{0+}_{C_\pm,1}(X,d).
\end{equation*}

For later applications, we fix
a point $y_-\in C_-$ (resp. $y_+\in C_+$) away from $x_-$ (resp. $x_+$).
Let $Q^{0+,0+}_{C_\pm,1|1}(X,d_\pm)$ be the moduli stacks of $(0+,0+)$-stable quasimap with one parametrized component
and $(1,1)$-weighted markings, as defined in Definition~\ref{def6.7}.
Here we assume the light marking is at $y_\pm$ and the heavy marking is contracted to
$x_\pm$ via the parametrization map from source curves to $C_\pm$.
Since the point $y_\pm$ is fixed,
adding a light point on the parametrized component at $y_\pm$ does not change the moduli space. Hence, we have:
\begin{equation}
	\label{eq:pm_moduli-new}
	\ca{M}_\pm^{d} \cong Q^{0+,0+}_{C_\pm,1 | 1}(X,d).
\end{equation}

\subsection{Decomposition of the central fiber}
Let $\fr{C}_0,\fr{B}_0$ and $\mdo$ denote the fibers of
$\fr{C},\fr{B}$ and $\md$ over $0\in B$, respectively.
In this section, we describe the structure of the central fiber $\mdo$.

Set
\[
	\fr{B}_0[n]:=\ca{T}\times \underbrace{\ca{T}_\sim\times\dots\times \ca{T}_\sim}_\text{$n$ copies}\times \ca{T}	.
\]
Then there is a natural map
\[
\mathrm{gl}_n: 
\fr{B}_0[n]
\rightarrow \fr{B}_0
\]
induced by gluing the expanded relative pairs of $(C_\pm,x_\pm)$ and
non-rigid expanded degenerations associated to the node.
By Remark 4.7 and the proof of Theorem 4.9 in~\cite{Qu} (see also~\cite[Proposition 2.19]{Li-Wu} and~\cite[Lemma 3]{Lee}),
$\fr{B}_0\subset \fr{B}$ is a complete intersection substack, $\fr{B}_0[n]$ is the normalization of the codimension-$n$
strata in $\fr{B}_0$, and we have
\begin{equation}
	\label{eq:normalization-structure-sheaf}
\sum_{n=0}^\infty (-1)^n (\mathrm{gl}_n)_* \O_{\fr{B}_0[n]}
=\O_{\fr{B}_0}.	
\end{equation}

There are evaluation maps
\[
\ev_\pm: \mdpm \rightarrow X
\quad
\text{and}
\quad
(\ev_-^\sim,\ev_+^\sim):
\mdt\rightarrow X\times X	
\]
defined by the restrictions of the universal quotients to $\fr{x}_\pm$.
Given $d,n\geq 0$, we define 
\[
\Lambda_d^n	=\Big\{
	\underline{d} = (d_-,d_1,\dots, d_n,d_+):
	d_-+d_++\sum_{i=1}^n d_i = d, d_\pm\geq 0, d_i>0
\Big\}.
\]
Given $\underline{d}\in \Lambda_d^n$, we set
\[
	\ca{M}^{\underline{d}}_0[n]
	:=
	\ca{M}^{d_-}_-\times_X \ca{M}^{d_1}_\sim \times_X
	\dots \ca{M}^{d_n}_\sim \times_X\ca{M}^{d_+}_+.
\]
There is a gluing map $\iota_{\ud}:	\ca{M}^{\underline{d}}_0[n]
\rightarrow \ca{M}_0^d$ of stable quotients, which fits into the following
cartesian diagram
\begin{equation}
	\label{eq:gluing-diagram}
	\begin{tikzcd}
		\coprod\limits_{ \ud\in\Lambda_d^n  }\ca{M}_0^{\underline{d}}[n] \arrow[r] \arrow[d] & \ca{M}_0^d \arrow[d] \\
	\fr{B}_0[n] \arrow[r, "\mathrm{gl}_n"] & \fr{B}_0
	\end{tikzcd}
\end{equation}
For $n\geq 0$, we think of $\coprod_{ \ud\in\Lambda_d^n}\ca{M}_0^{\underline{d}}[n]$
as the normalization of the codimension-$n$ strata in $\ca{M}_0^d$.

\subsection{Decomposition of the virtual structure sheaf}	
Let $\ca{S}_{\ca{M}_0^d}$ and $\ca{Q}_{\ca{M}_0^d}$ be the universal subbundle and quotient sheaf
over the universal curve $\pi_d:\ca{M}_0^d\times_{\fr{B}_0}\fr{C}_0 \rightarrow\ca{M}_0^d$,
whose pullbacks to $\ca{M}_0^{\ud}[n]$ via $\iota_{\ud}$ are denoted by
$\ca{S}_{\ca{M}_0^{\ud}[n]}$ and $\ca{Q}_{\ca{M}_0^{\ud}[n]}$, respectively.
We define the relative POT of $\ca{M}_0^{\ud}[n]\rightarrow\fr{B}_0[n]$
by pulling back the relative POT
$
(
R\shom_{\pi_{d}}(\ca{S}_{\ca{M}_0^d},\ca{Q}_{\ca{M}_0^d})
)^\vee
$
of $\ca{M}_0^d\rightarrow \fr{B}_0$ via $\iota_{\ud}$:
\[
	\left(R\shom_{\pi_{\ud }}(\ca{S}_{\ca{M}_0^{\ud}[n]},\ca{Q}_{\ca{M}_0^{\ud}[n]})
	\right)^\vee,
\]
where $\pi_{\ud}:\fr{C}_0[n]\rightarrow \ca{M}_0^{\ud}[n]$ is the universal curve.
These two relative POTs induce virtual pullbacks for the two vertical maps in~\eqref{eq:gluing-diagram},
and we define the virtual structure sheaves $\ovir_{\ca{M}_0^d}$ and $\ovir_{\ca{M}_0^{\ud}[n]}$
as the virtual pullbacks of $\O_{\fr{B}_0}$ and $\O_{\fr{B}_0[n]}$, respectively.
According to~\cite[Theorem 2.6]{Qu}, virtual pullbacks commute with Gysin pullbacks. Hence, we have
\begin{equation}
	\label{eq:ovir-pullback}
	(\mathrm{gl}_n)^!
	\ovir_{\ca{M}_0^d}
	=	\ovir_{\ca{M}_0^{\ud}[n]}.
\end{equation}
If we consider the standard $T=(\C^*)^N$-action on $\ca{M}_0^d$ and $\ca{M}_0^{\ud}[n]$,
then the above equation also holds in the $T$-equivariant $K$-group.

\begin{theorem}
	\label{thm:virtual-normalization}
	We have 
	\begin{equation*}
		\sum_{n\geq 0 }	
	\sum_{\underline{d}\in \Lambda_d^n}
	(-1)^n (\iota_{\ud})_*\ovir_{	\ca{M}^{\underline{d}}_0[n]}
	=\ovir_{\ca{M}_0^d}
	\end{equation*}
 in $\kG(\ca{M}_0^d)$.
	\end{theorem}
	\begin{proof}
		The theorem follows~\eqref{eq:normalization-structure-sheaf},~\eqref{eq:ovir-pullback}, and the commutativity
		between virtual pullbacks and proper pushforwards~\cite[Proposition 2.4]{Qu}.
	\end{proof}

To split invariants over $\ca{M}^{\underline{d}}_0[n]$, we first define the relative POTs for $\ca{M}_\pm^{d_\pm}\rightarrow \ca{T}$ and $\ca{M}_\sim^{d_i}\rightarrow \ca{T}_\sim$
using~\eqref{eq:alternate-POT} (or equivalently~\eqref{eq:POT}).
Note that the ($T$-equivariant) isomorphisms~\eqref{eq:rational-bridge} and~\eqref{eq:pm_moduli-new}
preserve these relative POTs. Let $\ovir_{\ca{M}^{d_\pm}_\pm}$ and $\ovir_{\ca{M}^{d_i}_\sim}$
be the virtual pullbacks of the structure sheaves of $\ca{T}$ and $\ca{T}_\sim$, respectively.
Consider the cartesian diagram:
\begin{equation*}
	\begin{tikzcd}
\ca{M}_0^{\underline{d}}[n] \arrow[r] \arrow[d] & 
\ca{M}^{d_-}_-\times\ca{M}^{d_1}_\sim \times
\cdots \ca{M}^{d_n}_\sim\times\ca{M}^{d_+}_+
\arrow[d] \\
\fr{B}_0[n] \times X^{\times n} \arrow[r, "\Delta_X^{\times n}"] & 
\fr{B}_0[n] \times (X\times X)^{\times n}
	\end{tikzcd}.	
\end{equation*}
Let $\fr{C}'_0[n]$ denote the pullback of the univeral curve over $\ca{M}^{d_-}_-\times\ca{M}^{d_1}_\sim \times\cdots \ca{M}^{d_n}_\sim\times\ca{M}^{d_+}_+$
to $\ca{M}_0^{\underline{d}}[n] $. Then we have a commutative diagram
\begin{equation*}
	\begin{tikzcd}
		\fr{C}'_0[n]		\arrow[r,"p"] \arrow[dr,"f'"'] & 
		\fr{C}_0[n]
\arrow[d,"f"] \\
& 
\fr{X}
	\end{tikzcd}	
\end{equation*}
where $f$ and $f'$ are the universal maps and $p$ is the gluing map.
By the stability condition of the quotients, $Lf^*\bb{T}_{\fr{X}}=\shom(\ca{S}_{\ca{M}_0^{\ud}[n]},\ca{Q}_{\ca{M}_0^{\ud}[n]})$
is locally free at the nodes of the universal curve.
It follows from the arguments in the proof of the cutting-edges axiom in~\cite{Behrend} that
the relative POTs of $\ca{M}_0^{\underline{d}}[n]$ and $\ca{M}^{d_-}_-\times\ca{M}^{d_1}_\sim \times\cdots \ca{M}^{d_n}\times\ca{M}^{d_+}_+$ are compatible with respect to $\Delta_X^{\times n}$
(see~\cite{Behrend-Fantechi}).
Then by the functoriality of virtual pullbacks~\cite[Proposition 2.11]{Qu}, we have
\begin{equation}
	\label{eq:split-virtual-structure-sheaf}
	\ovir_{\ca{M}_0^{\underline{d}}[n]} = \left(\Delta_X^{\times n}\right)^! 
	\left(
		\ovir_{\ca{M}^{d_-}_-}\boxtimes\ovir_{\ca{M}^{d_1}_\sim} \boxtimes
\dots \ovir_{\ca{M}^{d_n}_\sim}\boxtimes\ovir_{\ca{M}^{d_+}_+}
	\right)
\end{equation}
in $\kG(\ca{M}_0^{\underline{d}}[n])$.

\subsection{Degeneration formula}

We choose a point $y_-\in C_-$ (resp. $y_+\in C_+$) away from $x_-$ (resp. $x_+$).
After shrinking $B$ if necessary, we extend these two markings to sections $\fr{y}_\pm:\fr{B}\rightarrow \fr{C}$
of the expanded degeneration of curves.

Consider the `stacky' evaluation map
\[
	\wev_{y_\pm}: \ca{M}_\pm^{d_\pm}\rightarrow \fr{X}
\]
induced by the restriction of the universal subbundles at $y_\pm$.
By abuse of notation, we also use $\wev_{y_\pm}$ to denote the evaluation maps on $\ca{M}_0^{\underline{d}}[n]$ and $\ca{M}^d_0$.
For $\mathrm{V},\mathrm{W}\in \kg(\fX)$, we define
\[
	\langle\!\langle
	\mathrm{V},
	\mathrm{W}
	\rangle\!\rangle^{T}_{\deg}
	:=\sum_{d\geq 0} q^d\chi^{\vir, T}
	\left(
		\ca{M}_0^{d},	
	\wev_{y_-}^{*}
	\left(
	\mathrm{V}
	\right)
	\cdot
	\wev_{y_+}^{*}
	\left(
	\mathrm{W}
	\right)
	\right).
	\]
	It follows from Theorem~\ref{thm:virtual-normalization} that
\begin{align*}
	\langle\!\langle
	\mathrm{V},
	\mathrm{W}
	\rangle\!\rangle^{T}_{\deg}
	=\sum_{n\geq 0 } (-1)^n
	\sum_{d\geq 0,\ud\in \Lambda_d^n }q^d
	\chi^{\vir, T}
	\left(
		\ca{M}_0^{\underline{d}}[n],	
	\wev_{y_-}^{*}
	\left(
	\mathrm{V}
	\right)
	\cdot
	\wev_{y_+}^{*}
	\left(
	\mathrm{W}
	\right)
	\right).
\end{align*}
Define $F^{\epsilon=0+}_{\alpha,\beta} = g_{\alpha,\beta}+F_{\alpha,\beta}'$, where
$g_{\alpha,\beta}:=\chi^T(X,\O_\alpha\cdot\O_\beta)$ and
\[
	F_{\alpha,\beta}'=	\sum_{d> 0} q^{d}
	\langle
		\O_\alpha,\O_\beta
	\rangle_{0,2,d}^{\epsilon=0+,T}.
\]
Using~\eqref{eq:split-virtual-structure-sheaf} and the isomorphisms~\eqref{eq:rational-bridge},~\eqref{eq:pm_moduli-new}, we obtain
\begin{multline*}
	\sum_{d\geq 0,\ud\in \Lambda_d^n }q^d
	\chi^{\vir, T}
	\left(
		\ca{M}_0^{\underline{d}}[n],	
	\wev_{y_-}^{*}
	\left(
	\mathrm{V}
	\right)
	\cdot
	\wev_{y_+}^{*}
	\left(
	\mathrm{W}
	\right)
	\right)
=	
\sum_{\alpha,\beta,a_i,b_i \in \P_{r,k}}
\langle\!\langle
\mathrm{V}
,
\O_\alpha
\rangle\!\rangle_{C_-,1| 1}^{0+,T}
\\
\cdot
g^{\alpha,a_1}
F_{a_1,b_1}'
g^{b_1,a_2}
\cdots
F_{a_{m},b_{m}}'
g^{b_{m},\beta}
\cdot
\langle\!\langle
	\O_\beta,
	\mathrm{W}
	\rangle\!\rangle_{C_+,1| 1}^{0+,T}
.
\end{multline*}
Here $g^{\alpha,\beta}$ denote the coefficients of the inverse of $(g_{\alpha,\beta})$.
It was first observed in~\cite{WDVV} (in quantum $K$-theory) that the combinatorics involved in the above equation
can be simplified using the inverse matrix
$F_{\epsilon=0+}^{\alpha,\beta}$, which equals to
\[
	g^{\alpha,\beta}+\sum_{m \geq 1, a_i,b_i\in \P_{r,k}}(-1)^m
	g^{\alpha,a_1}
F_{a_1,b_1}'
g^{b_1,a_2}
\cdots
F_{a_{m},b_{m}}'
g^{b_{m},\beta}.
\]
In sum, we obtain
\begin{equation*}
\langle\!\langle
	\mathrm{V},
	\mathrm{W}
	\rangle\!\rangle^{T}_{\deg}
=
\sum_{\alpha,\beta\in \P_{r,k}}
\langle\!\langle
\mathrm{V}
,
\O_\alpha
\rangle\!\rangle_{C_-,1| 1}^{0+,T}
\cdot
F_{\epsilon=0+}^{\alpha,\beta}
\cdot
\langle\!\langle
	\O_\beta,
	\mathrm{W}
	\rangle\!\rangle_{C_+,1| 1}^{0+,T}
.
\end{equation*} 
Hence by using the wall-crossing results from Proposition~\ref{prop:wall-crossing} and Theorem~\ref{prop:epsilon-wall-crossing}, we have 
\begin{equation}
	\label{eq:factorization}
	\langle\!\langle
		\mathrm{V},
		\mathrm{W}
		\rangle\!\rangle^{T}_{\deg}
	=
	\sum_{\alpha,\beta\in \P_{r,k}}
	\langle\!\langle
	\mathrm{V}
	,
	\rO_\alpha
	\rangle\!\rangle_{g_-}^{\quot,T}
	\cdot
	F^{\alpha,\beta}
	\cdot
	\langle\!\langle
		\rO_\beta,
		\mathrm{W}
		\rangle\!\rangle_{g_+}^{\quot,T}
		.
	\end{equation}

Let $q_{\fr{B}}:\ca{M}^d\rightarrow \fr{B}$ and
$q_B:\ca{M}^d\ra B$ be the good degeneration of Quot schemes.
Let $i_b:b\hookrightarrow C$ be the regular embedding of a closed point.
Consider the cartesian diagram
\[
	\begin{tikzcd}
	\ca{M}^d_b \arrow[r] \arrow[d, "\iota_b"] & \fr{B}_b \arrow[r] \arrow[d] & b \arrow[d,"i_b"]\\
	\ca{M}^d \arrow[r,"q_{\fr{B}}"] & \fr{B} \arrow[r] & B
	\end{tikzcd}
	\]
We define the virtual structure sheaf of the good degeneration by
\[
	\ovir_{\ca{M}^d}: =q^!(\O_{\fr{B}})\in \kG(\ca{M}^d)	.
\]
Since the POT of $\ca{M}^d_b $ is the pullback of the relative POT of $q$,
we have
\[
i_b^!(\ovir_{\ca{M}^d}) = \ovir_{\ca{M}^d_b }.
\]
Let $F\in \kg(\ca{M}^d)$. Then
\begin{align*}
\chi^{\vir,T}\left(\ca{M}^d_b, F |_{\ca{M}^d_b}\right) 
=	
\chi^T\left(\ca{M}^d,F\cdot (\iota_b)_*i_b^!(\ovir_{\ca{M}^d} )\right)
=
\chi^T\left(B,i_b^!(q_B)_*(F\cdot \ovir_{\ca{M}^d} )\right)
\end{align*}
is independent of $b$.
In particular, if we fix a point $b\neq 0$ in $B$, then
\begin{equation}
	\label{eq:deformation-invariance}
	\chi^{\vir,T}\left(\ca{M}^d_b, F |_{\ca{M}^d_b}\right) 
	=\chi^{\vir,T}\left(\ca{M}^d_0, F |_{\ca{M}^d_0}\right). 
\end{equation}

\begin{proof}[Proof of Theorem~\ref{thm:intro_degeneration_formula} (i)]
	Let $C=\fr{C}_b$ be the fiber of $\fr{C}\rightarrow B$ at $b\neq 0$.
	Then $C$ is a smooth curve of genus $g$, and we have 
	\[
		\ca{M}^d_b\cong \QuotC.
	\]
	Set $F=\wev_{\fr{y}_-}^{*}
	\left(
	\mathrm{V}
	\right)
	\cdot
	\wev_{\fr{y}_+}^{*}
	\left(
	\mathrm{W}
	\right),
	$
	where $\wev_{\fr{y}_\pm}:\ca{M}^d\rightarrow \fr{X}$ denote the
	stacky evaluation maps at $\fr{y}_\pm$.
	Then the degeneration formula~\eqref{eq:degeneration-1} follows from the
	decomposition~\eqref{eq:factorization}
	and the deformation invariance of virtual Euler characteristics~\eqref{eq:deformation-invariance}.
\end{proof}
\begin{rem}
	The degeneration formula~\eqref{eq:degeneration-2} in Theorem~\ref{thm:intro_degeneration_formula}
	(ii) can be proven similarly by analyzing a simple degeneration
	$\ca{C}'\rightarrow B$ over a smooth pointed curve $(B,0)$.
	Here, the fibers outside 0 are smooth curves of genus $g$,
	while the fiber at 0, $C'_{\deg}$, is an irreducible curve with one node.
	The associated stack of expanded degenerations,
	$\fr{C}'\rightarrow \fr{B}$, has central fibers formed by
	inserting a chain of $\bP^1$s (of various lengths) at the unique node of $C'_{\deg}$ (see~\cite[\textsection{2.2}]{Li-Wu}).
	The good degeneration of Quot schemes, $\QuotCPB\ra B$, parametrizes flat families of
	stable quotients of rank $N-r$ and degree $d$ over the universal family
	$\fr{C}'\rightarrow \fr{B}$.

	Let $\tilde{C}$ be the normalization of $C'_{\deg}$, with $x_-$ and $x_+$ in $\tilde{C}$ representing the preimages of the unique node in $C'_{\deg}$.
	The central fiber of this degeneration of Quot schemes is `virtually' a complete intersection
	substack, and its `normalization' is
	\begin{equation}
		\label{eq:normalization-central-fiber}
\coprod\limits_{ n,\ud\in\tilde{\Lambda}_d^n }
		\tilde{\ca{M}}_0^{d_0}
		\times_X \ca{M}^{d_1}_\sim \times_X
		\dots \ca{M}^{d_n}_\sim 
		\times_X
		\big(\tilde{\ca{M}}_0^{d_0}\big),
	\end{equation}
	where $\tilde{\Lambda}_d^n	=\{
		\underline{d} = (d_0,d_1,\dots, d_n):
		d_0+\sum_{i=1}^n d_i = d\ \text{and}\ d_i>0\ \text{for}\ 1 \leq i\leq n\}$,
	\[
		\tilde{\ca{M}}_0^{d_0}\cong Q^{0+,0+}_{\tilde{C}, 1 | 2}(X,d_0),
	\]
	and we put the further right $\tilde{\ca{M}}_0^{d_0}$ in parenthesis
	to indicate is is the same as the one on the left.
	Here, a source curve in $Q^{0+,0+}_{\tilde{C}, 2 | 2}(X,d_0)$ is formed
	by attaching rational tails to $\tilde{C}$ at $x_\pm$ and
	placing heavy markings on the outermost components of these rational tails.
	The inclusion of one light marking on the parametrized component $\tilde{C}$
	does not alter the moduli space but allows us to define invariants with one insertion in $\kg(\fX)$.
	Similar to the proof of Theorem~\ref{thm:intro_degeneration_formula} (i),
	we apply the deformation invariance of virtual Euler characteristics to relate $\langle\!\langle \mathrm{V} \rangle\!\rangle^{\quot,T}_{g}$
	to the alternating sum of virtual Euler characteristics over the components in~\eqref{eq:normalization-central-fiber}.
	This sum can then be rewritten as the right side of the degeneration formula~\eqref{eq:degeneration-2}
	using the wall-crossing formulas in Proposition~\ref{prop:wall-crossing} and Theorem~\ref{prop:epsilon-wall-crossing}.	
\end{rem}
	

\bibliographystyle{alphnum}
\bibliography{ref1}

\end{document}